\documentclass[notitlepage,reqno,11pt]{amsart}
\usepackage{latexsym,amssymb, epsfig, amsmath,amsfonts, subfigure,amsthm}

\usepackage{rotating}
\usepackage[toc,page]{appendix}
\usepackage{color}
\usepackage{multirow}
\usepackage{relsize}
\usepackage{microtype}

\usepackage{bm}

\usepackage[colorlinks, allcolors=blue]{hyperref}

\usepackage[margin=1in]{geometry}

\numberwithin{equation}{section}
 \newtheorem{assumption}{Assumption}[section]
\newtheorem{lemma}{Lemma}[section]
\newtheorem{theorem}{Theorem}[section]

\newtheorem{coro}{Corollary}[section]

\newtheorem{remark}{Remark}[section]

\newlength{\defbaselineskip}
\newcommand{\setlinespacing}[1]%
           {\setlength{\baselineskip}{#1 \defbaselineskip}}
\newcommand{\doublespacing}{\setlength{\baselineskip}%
                           {1 \defbaselineskip}}

\newcommand{\RR}{{\mathbb R}}

\def\E{\mathbb{E}}
\def\P{\mathbb{P}}

\newcommand{\sF}{{\mathcal{F}}}

\newcommand{\ep}{\epsilon}

\newcommand{\eps}  {\varepsilon}

\newcommand{\R}  {\mathbb{R}}
\newcommand{\N}  {\mathbb{N}}

\newcommand{\bone}{{\mathbf 1}}

\newcommand{\qforq}{\quad\mbox{for}\quad}

\newcommand{\qasq}{\quad\mbox{as}\quad}
\newcommand{\qinq}{\quad\mbox{in}\quad}

\newcommand{\non}{\nonumber}

\newcommand{\ttl}{\Large Epidemic models with varying infectivity}

\begin{document}

\title[]{\ttl}

\author[Rapha\"el \ Forien]{Rapha\"el Forien}
\address{INRAE, Centre INRAE PACA, Domaine St-Paul - Site Agroparc
84914 Avignon Cedex
FRANCE}
\email{raphael.forien@inrae.fr}

\author[Guodong \ Pang]{Guodong Pang}
\address{The Harold and Inge Marcus Department of Industrial and
Manufacturing Engineering,
College of Engineering,
Pennsylvania State University,
University Park, PA 16802 USA }
\email{gup3@psu.edu}

\author[{\'E}tienne \ Pardoux]{{\'E}tienne Pardoux}
\address{Aix--Marseille Universit{\'e}, CNRS, Centrale Marseille, I2M, UMR \ 7373 \ 13453\ Marseille, France}
\email{etienne.pardoux@univ.amu.fr}


\begin{abstract} 
\doublespacing
We introduce an epidemic model with varying infectivity and general exposed and infectious periods, where the infectivity of each individual  is a random function of the elapsed time since infection, those function being i.i.d. for the various individuals in the population. 
This approach models infection-age dependent  infectivity, and extends the classical SIR and SEIR models. 
We focus on the infectivity process  (total force of infection at each time), and prove a functional law of large number (FLLN). In the deterministic limit of this FLLN, the evolution of the mean infectivity  and of the proportion of susceptible individuals are determined by a two-dimensional deterministic integral equation.  From its solutions, we then obtain expressions for the evolution of the proportions of exposed,  infectious and recovered individuals. 
For the early phase, we study the stochastic model directly by using an approximate (non--Markovian) branching process, and show that the epidemic grows at an exponential rate on the event of 
non-extinction, which matches the rate of growth derived from the deterministic linearized  equations. 
We also use these equations to derive the expression for the basic reproduction number $R_0$ during the early stage of an epidemic, in terms of the average individual infectivity function and the exponential rate of growth of the epidemic, and apply our results to the Covid--19 epidemic. 
\end{abstract}


\maketitle



\doublespacing

\allowdisplaybreaks

\section{Introduction}

Most of the literature on epidemic models is based upon ODE models which assume that the length of time during which a given individual is infectious follows an exponential distribution. More precisely, those deterministic models are law of large numbers limits, as the size of the population tends to infinity,  of stochastic models where all transitions from one compartment to the next have exponential distributions, see \cite{britton2018stochastic} for a recent account. However, it is largely recognized that for most diseases, the durations of the exposed and infectious periods are far from following an exponential distribution. In the case of influenza, a deterministic duration would probably be  a better approximation. Recently in \cite{PP-2020}, the last two authors of the present paper have obtained the functional law of large numbers (FLLN) limits for SIS, SIR, SEIR and SIRS models where in the stochastic model the duration of the stay in the I compartment (resp. both in the E and the I, resp. both in the I and the R  compartments) follow a very arbitrary distribution. 
Of course, in this case  the stochastic model is not a Markov model, which makes some of the proofs more delicate. Indeed, the fluctuating part of a Markov process is a martingale, and many tools exist to study tightness and limits of martingales, which are missing in the non--Markovian setting. Nevertheless, we were able in \cite{PP-2020} to use \textit{ad hoc} techniques in order to circumvent that difficulty, and we proved not only FLLNs, but also functional central limit theorems (FCLTs). While the classical ``Markovian'' deterministic models are ODEs, our more general and more realistic ``non--Markovian'' deterministic models are Volterra type integral equations of the same dimension as the classical ODE models, i.e., equations with memory. Recently in \cite{FPP2020a}, the authors used the approach in \cite{PP-2020} to describe the Covid-19 epidemic in France. The flexibility of the choice for the law of the infectious period was very helpful in order to write a realistic model with very few compartments, and our model follows better the data than Markov models.

The aim of the present paper is to go a step further in the direction of realistic models of epidemics, and to consider the case where the infectivity of infectious individuals depends upon their  time since infection. It has been established in \cite{he2020temporal} that in the case of the Covid-19 disease, the infectivity of infectious individuals decreases after symptom onset. In fact it is believed that in most infectious diseases, the infectivity of infectious individuals depends upon the time since infection. This was already argued almost a century ago by Kermack and McKendrick, two of the founders of epidemic modeling in \cite{KMK}.  In that paper, the authors assume both an infection age infectivity, and an infection age recovery rate. The latter can be thought of as the hazard function of the duration of the infectious period, which then is a general absolutely continuous distribution.  Like in the present paper, their model is a Volterra integral equation. The same deterministic model has also been described as an ``age of infection epidemic model'' in \cite{brauer2008age} and in the recent book \cite[Chapter 4.5]{BCF-2019}.   See also two recent papers in the study of Covid-19 pandemic \cite{Gaubert,foutel-rodier_individual-based_2020}, which use a transport PDE model (it is worth noting that PDEs have been commonly used to capture the effect of age of infection in the epidemic literature, see, e.g.,  \cite{hoppensteadt1974age,thieme1993may,inaba2004mathematical,magal2013two}). The novelty of the present paper is that we prove that our integral equation deterministic model it is the law of large numbers limit of a well specified individual based stochastic model. 

The most realistic assumption is probably that this infectivity first increases continuously from $0$, and then decreases back to $0$. We shall however allow jumps in the random infectivity function, in order in particular to include the classical case of a constant infectivity during the infectious period. We also want to allow a very arbitrary law for the infectious (or exposed/infectious) period(s), as was done in \cite{PP-2020}. In this work again, the FLLN limiting deterministic model is a Volterra type integral equation, which is of the same dimension as the corresponding classical ODE model, see Theorem \ref{thm-FLLN}. We treat only the case of SIR and SEIR models (see also Remark \ref{rem-SIS} on the SIS and SIRS models), but we intend to extend in later publications our approach to other types of models, including models with age classes and spatial distribution, see already \cite{PP-2020b} for multi--patch models with general exposed and infectious durations. We have also established in a separate publication the FCLT associated to the FLLN established in the present paper, see \cite{PP-2020c}.

Our approach in this paper is to assume that in the original stochastic finite population model, the infectivity of each individual  is a random function of the time elapsed since his/her infection, those functions associated to various individuals being independent and identically distributed (i.i.d.). The total force of infection at each time is the aggregate infectivity of all the individuals that are currently infectious.  We assume that the infectivity random functions are piecewise continuous with a finite number of discontinuities, which includes all the commonly seen examples, in particular, constant infectivity over a given time interval as a special case.  They are also allowed to start with a value zero for a period of time to generalize the SEIR model.  These random functions then determine the durations of the exposed and infectious periods, and therefore, their corresponding probability distributions, which can be very general. 

Under the i.i.d. assumptions of these infectivity random functions of the various individuals, 
we prove a FLLN for the infectivity process, together with the counting processes for the  susceptible, exposed, infectious and recovered individuals.  The mean infectivity and the proportion of susceptible individuals in the limit are uniquely determined by a two-dimensional Volterra integral equation. Given these two functions, the proportions of exposed, infectious and recovered individuals in the limit are expressed in terms of the two above quantities.  They generalize the integral equations in the standard SIR/SEIR models with general exposed and infectious periods in \cite{PP-2020}. Our proofs are based upon Poisson random measures associated with the infectivity process, which help us to establish tightness and convergence. This paper further develops the techniques in \cite{PP-2020}, since for establishing the mean infectivity equation, we cannot integrate by parts as was done in \cite{PP-2020}. See below Lemmas 4.4 and 4.5, which give a key argument for the proof of Lemma 4.6.

Our limiting integral equations can be easily solved numerically. 
 For the standard SIR/SEIR model with general exposed and infectious periods, the integral equations are implemented to estimate the state of the Covid-19 pandemic in France in \cite{FPP2020a}. In another recent work, Fodor et al.  \cite{FodorKatzKovacs} argue  that integral equations (in the case of deterministic infectious periods) should be used instead of ODEs since the latter may significantly underestimate the initial basic reproduction number $R_0$. We claim that our model may be used to better predict the trajectory of the epidemic, especially at the beginning of the epidemic and when certain control measures like lockdown and reopening are implemented. 

 We also study the early phase of the epidemic, during which the proportion of susceptible individuals remains close to $1$, which allows to linearize the system of equations. However, typically the epidemic starts with a very small number of infected individuals, so that we need to go back to the stochastic model if we want to describe that early phase. Thanks to a comparison with (non--Markov) branching processes, we are able to show that, conditioned upon non-extinction,  the epidemic grows at an exponential rate $\rho$, reaching a given proportion of infected individuals in the population after a length of time of the order of $\rho^{-1}\log(N)$, if $N$ is the total population size. After that time, we can follow the linearized deterministic model, whose rate of growth is the same $\rho$.

The rate $\rho$ is easily estimated from the data (if $d$ denote the ``doubling time'', i.e., the number of days necessary for the number of cases to double, $\rho=d^{-1}\log(2)$). It is then interesting to express
 the basic reproduction number $R_0$ in terms  of $\rho$ and of
 the average infectivity function, a formula which we deduce from the linearized Volterra equation, as was 
 already done  by \cite{wallinga_how_2007}, see their formula (2.7).  
 We compute explicitly the value of $R_0$ for different values of two unknown parameters for the case of the early phase of the Covid--19 epidemic in France, assuming a decrease of the infectivity compatible with the results in \cite{he2020temporal}. We see that the decrease of the infectivity with infection--age induces a decrease of $R_0$.
 
 The paper is organized as follows. In Section \ref{subsec:model}, we formulate our stochastic model, and make precise all the assumptions. In Section \ref{sec:flln}, we state  the FLLN,  Theorem \ref{thm-FLLN}. Section \ref{sec:early} is devoted to the early phase of the epidemic: we state Theorem \ref{thm:early_phase} which describes the behavior of the stochastic model, and Theorem \ref{thm:early_phase_deterministic}, which describes the behavior of the deterministic linearized model. In Section \ref{sec:estimating_contact_rate}, we express $R_0$ in terms of the exponential growth rate and the mean infectivity function, and in Section \ref{subsec:covid} we apply our techniques to the French Covid--19 epidemic during 2020. Section \ref{sec:early_phase} is devoted to the proof of Theorem \ref{thm:early_phase} and Theorem \ref{thm:early_phase_deterministic}, and Section \ref{sec:proofFLLN} to the proof of Theorem \ref{thm-FLLN}.

\section{Model and Results}

\subsection{Model description} \label{subsec:model}
 All random variables and processes are defined in a common complete probability space $(\Omega, \mathcal{F}, \P)$. 
We consider a generalized SEIR epidemic model where each infectious individual has an infectivity that is randomly varying with the time elapsed since infection. 
As usual, the population consists of four groups of individuals, susceptible, exposed, infectious and recovered. 
Let $N$ be the population size, and $S^N(t), E^N(t), I^N(t), R^N(t)$ denote the sizes of the four groups, respectively. 
We have the balance equation $N= S^N(t) + E^N(t) + I^N(t) + R^N(t)$ for $t\ge 0$. Assume that $R^N(0)=0$, $S^N(0)>0$ and $E^N(0)+I^N(0) >0$ such that $S^N(0)+E^N(0) +I^N(0)=N$. Let $A^N(t)$ be the cumulative number of individuals that become infected in $(0,t]$ for $t\ge 0$ and denote the associated event times by $\tau^N_i$, $i=1,\dots, A^N(t)$. 

Note that an infected individual is either exposed or infectious. More precisely, he/she is first exposed, then infectious. Let us first consider those individuals who are infected after time $0$ (i.e. they are in the S compartment at time $0$). The $i$--th infected individual is infected at time $\tau^N_i$. He/she is first exposed during the time interval $[\tau^N_i,\tau^N_i+\zeta_i)$. Then he/she is infectious during the time interval $(\tau^N_i+\zeta_i,\tau^N_i+\zeta_i+\eta_i)$,
and finally removed on the time interval $[\tau^N_i+\zeta_i+\eta_i,+\infty)$.
To this individual is attached an infectivity process $\{\lambda_i(t): t\ge0\}$, which is a random right--continuous function such that 
\begin{align} \label{eqn-lambda-def-1}
\lambda_i(t)\begin{cases}  = 0,&\text{if $0\le t<\zeta_i$},\\
> 0,&\text{if $\zeta_i<t<\zeta_i+\eta_i$},\\
= 0,&\text{if $t\ge\zeta_i+\eta_i$}.
\end{cases}
\end{align}
We shall formulate some assumptions on the functions $\lambda_i$ below. Let us just say for now that the collection of the functions $\{\lambda_i(\cdot)\}_{i\ge1}$  are i.i.d.
Since 
\begin{equation} \label{eqn-zeta-lambda}
\zeta_i=\inf\{t>0,\ \lambda_i(t)>0\},\quad\text{and }\zeta_i+\eta_i=\inf\{t>0,\ \lambda_i(r)=0,\ \forall r\ge t\},
\end{equation}
the collection of random vectors $(\zeta_i,\eta_i)_{i\ge 1}$ is also i.i.d.

Each initially exposed individual is associated with an infectivity process $\lambda^0_j(t)$, $j=1,\dots, E^N(0)$, with a c{\`a}dl{\`a}g path; the $\lambda^0_j$'s are assumed to be i.i.d.  and such that 
\begin{equation}\label{eqn-zeta-lambda-0}
\zeta^0_j=\inf\{t>0,\ \lambda^0_j(t)>0\}>0 \text{ a.s.}\quad\text{and }\zeta^0_j+\eta^0_j=\inf\{t>0,\ \lambda^0_j(r)=0,\ \forall r\ge t\}.
\end{equation}
Each initially infectious individual is associated with an infectivity process $\lambda^{0,I}_k(t)$, $k=1,\dots, I^N(0)$, with a c{\`a}dl{\`a}g path;  the $\lambda^{0,I}_k$'s are also assumed to be i.i.d. 
and such that 
\begin{equation}\label{eqn-zeta-lambda-0I}
\inf\{t>0,\ \lambda^{0,I}_k(t)>0\}=0 \text{ a.s.}\quad\text{and }\eta^{0,I}_k=\inf\{t>0,\ \lambda^{0,I}_k(r)=0,\ \forall r\ge t\}.
\end{equation}
We will write $(\zeta,\eta)$ (resp. $(\zeta^0,\eta^0)$, resp. $\eta^{0,I}$) for a vector which has the same law as $(\zeta_i,\eta_i)$ (resp. $(\zeta^0_j,\eta^0_j)$, resp. $\eta^{0,I}_k$).
Let $H(du,dv)$ denote the law of $(\zeta,\eta)$, $H_0(du,dv)$ that of $(\zeta^0,\eta^0)$ and $F_{0,I}$ the c.d.f. of $\eta^{0,I}$.
Moreover, we define 
\begin{align*}
\Phi(t)&:=\int_0^t\int_0^{t-u}H(du,dv)=\P(\zeta+\eta\le t),\ \Psi(t):=\int_0^t\int_{t-u}^{\infty}H(du,dv)=\P(\zeta\le t<\zeta+\eta),\\
\Phi_0(t)&:=\int_0^t\int_0^{t-u}H_0(du,dv)=\P(\zeta^0+\eta^0\le t),\ \Psi_0(t):=\int_0^t\int_{t-u}^{\infty}H_0(du,dv)=\P(\zeta^0\le t<\zeta^0+\eta^0),\\
F_{0,I}(t)&:=\P(\eta^{0,I}\le t)\, .
\end{align*}
We shall also write \[ H(du,dv)=G(du)F(dv|u),\quad H_0(du,dv)=G_0(du)F_0(dv|u),\]
i.e., $G$ is the c.d.f. of $\zeta$ and $F(\cdot|u)$ is the conditional law of $\eta$, given that $\zeta=u$, 
$G_0$ is the c.d.f. of $\zeta^0$ and $F_0(\cdot|u)$ is the conditional law of $\eta^0$, given that $\zeta^0=u$.
In the case of independent exposed and infectious periods, it is reasonable that the infectious periods of the initially exposed individuals have the same distribution as the newly exposed ones, that is, $F_{0}=F$. Note that 
$\Psi(t) = G(t) - \Phi(t)$ and $\Psi_0(t) = G_0(t) - \Phi_0(t)$. 
Also, let $G^c_0=1-G_0$, $G^c=1-G$, $F^c_{0,I}=1-F_{0,I}$, and $F^c=1-F$. 

We remark that our framework allows very general random infectivity functions $\lambda(t)$, which can be piecewise continuous (see Assumption \ref{AS-lambda}) and can also generate dependent and independent $\zeta$ and $\eta$ variables for each individual.   We give an example of independent $\zeta$ and $\eta$ variables. 
Let $\zeta$, $\eta$ and $h$ be random objects so that $\zeta$ is independent of the pair $(\eta,h)$, where $\zeta$ and $\eta$ are $\RR_+$ valued and $h$ is a random element of $C([0,1];\RR_+)$ satisfying $h(0)=h(1)=0$ and $h(t)>0$ for $0<t<1$, a.s. ($\eta$ and $h$ can be dependent).  We extend $h$ as an element of $C(\RR;\RR_+)$ by specifying that $h(t)=0$ if $t \notin [0,1]$.  Define $\lambda(t) = h(\zeta\eta^{-1}(\zeta^{-1}t-1))$ for any $t\ge0$. Then $\lambda(t) =0$ on $[0,\zeta]$, and again on $[\zeta+\eta,+\infty)$, where $\lambda(t) >0$ if $\zeta<t < \zeta+\eta$. By construction, $\zeta$ and $\eta$ are independent.

The total force of infection which is exerted on the susceptibles at time $t$ can be written as 
\begin{align} \label{eqn-mf-I}
\mathfrak{I}^N(t) = \sum_{j=1}^{E^N(0)}\lambda^0_j(t) +\sum_{k=1}^{I^N(0)}\lambda^{0,I}_k(t)
+\sum_{i=1}^{A^N(t)}\lambda_i (t-\tau^N_i) \,, \quad t \ge 0.
\end{align}
Thus, the instantaneous infectivity rate function at time $t$ is 
\begin{align} \label{eqn-Phi}
\Upsilon^N(t) = \frac{S^N(t)}{N} \mathfrak{I}^N(t), \quad t\ge 0. 
\end{align}
The infection process $A^N(t)$ can be expressed by 
\begin{align} \label{eqn-An-rep-1}
A^N(t)=\int_0^t\int_0^\infty{\bf1}_{u\le \Upsilon^N(s^-)}Q(ds,du), \quad t \ge 0,
\end{align}
where $Q$ is a standard Poisson random measure (PRM) on $\R^2_+$, and
we use ${\bf1}_{\{\cdot\}}$ for the indicator function.
One may observe that besides the PRM $Q$, 
the randomness in the epidemic dynamics comes only from 
the infectivity processes $\{\lambda^0_j(t)\}$, $\{\lambda^{0,I}_k(t)\}$ and $\{\lambda_i(t)\}$ (the infectious periods $\{\eta^0_j\}$, $(\eta^{0,I}_k)$ and $\{\eta_i\}$ are induced from them). 

The epidemic dynamics of the model can be described by 
\begin{align}
S^N(t) &\,=\,S^N(0)-A^N(t)\,, \\
E^N(t)&\,=\, \sum_{j=1}^{E^N(0)}{\bf1}_{\zeta^0_j>t}+\sum_{i=1}^{A^N(t)}{\bf1}_{\tau^N_i+\zeta_i>t}\,, \label{eqn-EN}\\
 I^N(t) &\,=\, \sum_{j=1}^{E^N(0)}{\bf1}_{\zeta^0_j\le t<\zeta^0_j+\eta^0_j}+\sum_{k=1}^{I^N(0)}{\bf1}_{\eta^{0,I}_k>t}+\sum_{i=1}^{A^N(t)}{\bf1}_{\tau^N_i+\zeta_i\le t<\tau^N_i+\zeta_i+\eta_i}\,, \label{eqn-IN}\\
  R^N(t) &\,=\, \sum_{j=1}^{E^N(0)}{\bf1}_{\zeta^0_j+\eta^0_j\le t}+\sum_{k=1}^{I^N(0)}{\bf1}_{\eta^{0,I}_k\le t}+\sum_{i=1}^{A^N(t)}{\bf1}_{\tau^N_i+\zeta_i+\eta_i\le t}\,. \label{eqn-RN}
\end{align}
In the case where $\zeta^0_j=0$ and $\zeta_i=0$,  the model is a generalized SIR model, and $E^N(t)\equiv0$.

We now make the following assumptions on the
infectivity functions and the initial quantities. 
We first state our assumptions on $\lambda^0$, $\lambda^{0,I}$ and $\lambda$.

\begin{assumption} \label{AS-lambda}
The random functions $\lambda(t)$ (resp. $\lambda^0(t)$ and resp. $\lambda^{0,I}(t)$  ), of which $\lambda_1(t), \lambda_2(t),\ldots$ (resp. $\lambda^0_1(t), \lambda^0_2(t), \ldots$ and resp. $\lambda^{0,I}_1(t), \lambda^{0,I}_2(t), \ldots$) are i.i.d. copies, satisfy the following assumptions.  
There exists a constant $\lambda^*<\infty$ such that $\sup_{t \in [0,T]} \max\{\lambda^0(t), \lambda^{0,I}(t), \lambda(t)\} \le \lambda^*$ almost surely, and in addition there exist a given number $k\ge1$, a random sequence $0=\xi^0\le\xi^1\le\cdots\le\xi^k=\eta$ and random functions $\lambda^j\in C(\R_+;\R_+)$, $1\le j\le k$
such that 
\begin{equation} \label{eqn-lambda}
 \lambda(t)=\sum_{j=1}^k\lambda^j(t){\bf1}_{[\xi^{j-1},\xi^j)}(t)\,.
\end{equation}
\end{assumption}
We define 
\[ \varphi_T(r):=\sup_{1\le j\le k}\sup_{0\le s, t\le T, |t-s|\le r}|\lambda^j(t)-\lambda^j(s)|\,.\]
It is clear that for each $T>0$, $\varphi_T$ is continuous and $\varphi_T(0)=0$.

Let $\bar{\lambda}^0(t) =\E[\lambda^0(t)]$, $\bar{\lambda}^{0,I}(t) =\E[\lambda^{0,I}(t)]$  and $\bar{\lambda}(t) =\E[\lambda(t)]$ for $t\ge 0$. 

It is clear that  $\bar{\lambda}^0(t), \bar{\lambda}^{0,I}(t)$ and $ \bar{\lambda}(t) $ are all c{\`a}d{\`a}g, and they are also uniformly bounded by $\lambda^*$.    
\begin{remark}
We think that $\lambda(t)$ being continuous is a good model of reality. However, the early phase of the function $\lambda(t)$ is not well known, since 
patients are tested only after symptom onset, and usually (this is the case in particular for the Covid--19) they may have been infectious (i.e., with $\lambda(t)>0$) prior to that. Consequently we should not exclude the possibility that $\lambda(t)$ jumps to its maximum at time $\zeta$, and the decreases continuously to $0$.

Moreover, in order to include the ``classical'' models where $\lambda(t)$ is first $0$ during the exposed period, and then equal to a positive constant during the infectious period, as well as possible models of infectivity that would be piecewise constant, we allow $\lambda(t)$ to have a given number of jumps. 
\end{remark}
\bigskip

For one of our results, we shall need the following assumption.
\begin{assumption} \label{AS-Lambda-Moment2}
Assume that
\begin{align*}
	\E\left[\left(\int_0^\infty\lambda(t)dt\right)^2\right]<\infty, && \E\left[\left(\int_0^\infty\lambda^0(t)dt\right)^2\right]<\infty.
\end{align*}
\end{assumption}
\begin{remark}
The assumption on the second moment of $\int_0^\infty \lambda(t)dt$ will be necessary in order to apply Theorem 3.2 from \cite{crump_general_1969} to the branching process approximation of the stochastic model for the early phase of the epidemic. Since we assume that $\lambda(t)\le\lambda^\ast$, for this second moment condition to be satisfied, it is sufficient that the  duration of the infectious period $\eta$ satisfies $\E[\eta^2]<\infty$, which certainly is not a serious restriction in practice. In our application to the Covid--19 in Section~\ref{subsec:covid}, we choose a law with compact support for $\eta$.
\end{remark}

Let $\bar{X}^N := N^{-1}X^N$ for any process $X^N$. Let $D=D(\R_+; \R)$ denote the space of $\R$--valued c{\`a}dl{\`a}g functions defined on $\R_+$. Throughout the paper, convergence in $D$ means convergence in the  Skorohod $J_1$ topology, see Chapter 3 of \cite{billingsley1999convergence}. 
 Also, $D^k$ stands for the $k$-fold product equipped with the product topology. 

\begin{assumption} \label{AS-FLLN}
Assume that there exist deterministic constants $\bar{E}(0),\bar{I}(0) \in [0,1]$ such that $0<\bar{E}(0)+\bar{I}(0)<1$, and
 $(\bar{E}^N(0),\bar{I}^N(0))\to (\bar{E}(0),\bar{I}(0)) \in \RR^2_+$ in probability as $N\to\infty$. 
\end{assumption}

Finally we make the following independence assumption.
\begin{assumption} \label{indep}
Assume that the triple
$(\lambda_i(\cdot),i\ge1;\ \lambda^0_j(\cdot),j\ge1;\ \lambda^{0,I}_k(\cdot), k\ge1)$, $(E^N(0),I^N(0))$ and $Q$ (the PRM upon which the construction of the process $A^N(\cdot)$ is based) are independent.
\end{assumption}

\subsection{FLLN}\label{sec:flln}
We now state the main result of this paper.

\begin{theorem} \label{thm-FLLN}
 Under Assumptions \ref{AS-lambda}, \ref{AS-FLLN} and \ref{indep}, 
\begin{equation} \label{eqn-FLLN-conv}
\big(\bar{S}^N, \bar{\mathfrak{I}}^N, \bar{E}^N, \bar{I}^N, \bar{R}^N\big) 
\to \big(\bar{S}, \bar{\mathfrak{I}}, \bar{E}, \bar{I}, \bar{R}\big) \qinq D^5 \qasq N \to \infty,
\end{equation}
in probability, locally uniformly in $t$. 
The limits $\bar{S}$ and $\bar{\mathfrak{I}}(t)$ are the unique solution of the following system of Volterra integral equations
\begin{align}
\bar{S}(t)&=1- \bar{E}(0)-\bar{I}(0)-\int_0^t\bar{S}(s)\bar{\mathfrak{I}}(s)ds\,, \label{eqn-barS}\\
\bar{\mathfrak{I}}(t)&=\bar{E}(0)\bar{\lambda}^0(t) + \bar{I}(0)\bar{\lambda}^{0,I}(t)+\int_0^t\bar{\lambda}(t-s)\bar{S}(s)\bar{\mathfrak{I}}(s)ds\,,\label{eqn-barfrakI}
\end{align}
and the limit $(\bar{E},\bar{I},\bar{R})$ is given by the following integral equations:
\begin{align}
 \bar{E}(t) &=\bar{E}(0)G_0^c(t)+\int_0^tG^c(t-s)\bar{S}(s)\bar{\mathfrak{I}}(s)ds\,,\label{eqn-barE} \\
 \bar{I}(t) &= \bar{I}(0) F_{0,I}^c(t) + \bar{E}(0)\Psi_0(t)  +\int_0^t \Psi(t-s) \bar{S}(s)\bar{\mathfrak{I}}(s)ds\,,\label{eqn-barI}\\
\bar{R}(t)&=\ \bar{I}(0) F_{0,I}(t) + \bar{E}(0) \Phi_0(t)  +\int_0^t \Phi(t-s) \bar{S}(s)\bar{\mathfrak{I}}(s)ds\, . \label{eqn-barR}
\end{align}
The limit $\bar{S}$ is in $C$, and the limits $\bar{\mathfrak{I}}, \bar{E},\bar{I},\bar{R}$ are in $D$. If $\bar{\lambda}^0$ and $\bar{\lambda}^{0,I}$ are continuous, then $\bar{\mathfrak{I}}$ is in $C$, and if $G_0$ and $F_{0,I}$ are continuous, then   $ \bar{E},\bar{I},\bar{R}$ are in $C$. 

\end{theorem}

\begin{remark}
If we suppose only that Assumptions \ref{AS-FLLN} and \ref{indep} are valid, and $\sup_{t \in [0,T]} \max\{\lambda^0(t), \\ \lambda^{0,I}(t), \lambda(t)\} \le \lambda^*$ almost surely, then Theorem \ref{thm-FLLN} remains valid, but with the convergence in probability in $D^5$ being replaced by the convergence in probability in $L^p_{loc}(\R_+;\R^5)$, for any $p\ge1$.
\end{remark}

\begin{figure}[h]
	\centering
	\includegraphics[width=0.47\textwidth]{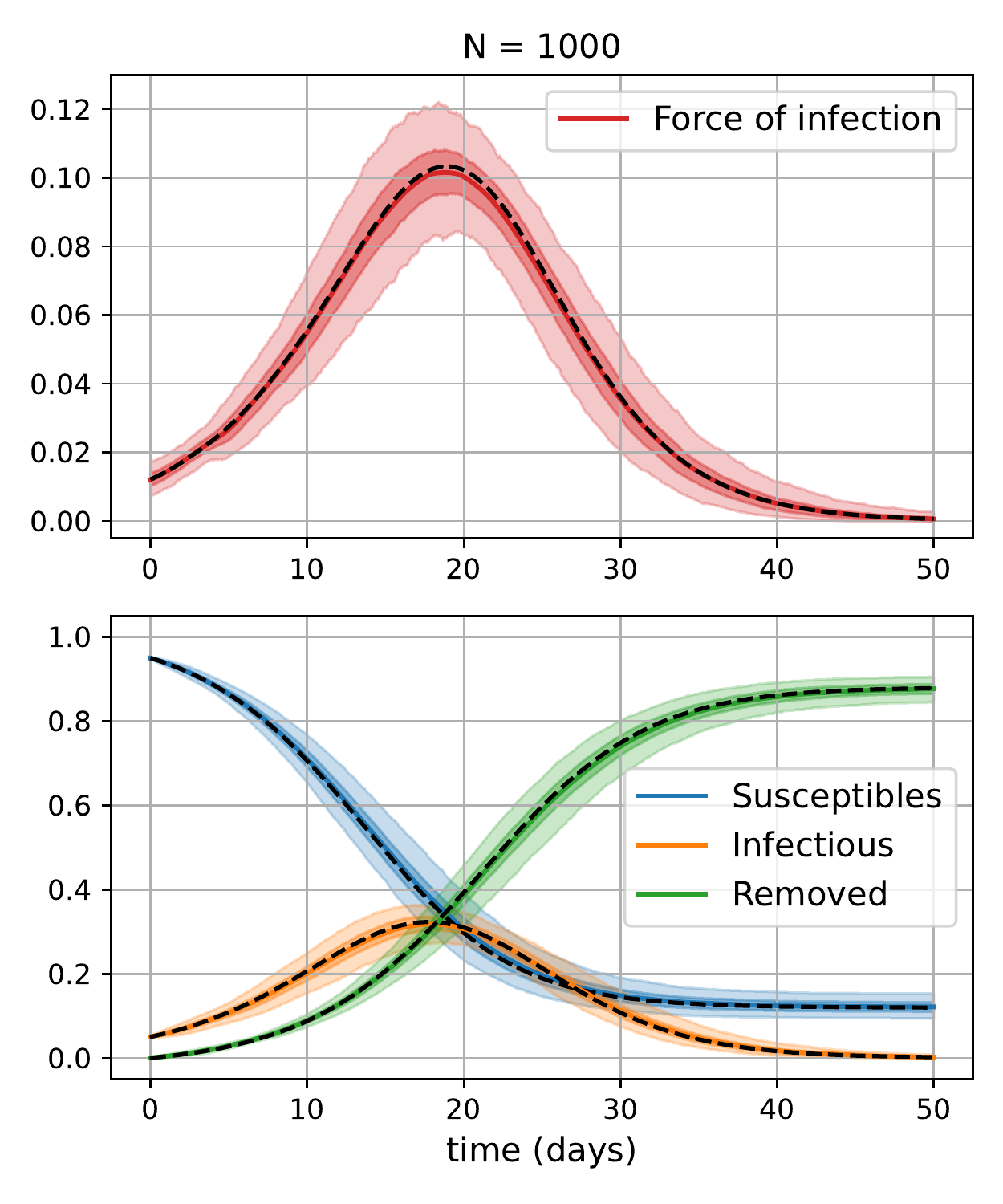}\hspace{0.02\textwidth}%
	\includegraphics[width=0.47\textwidth]{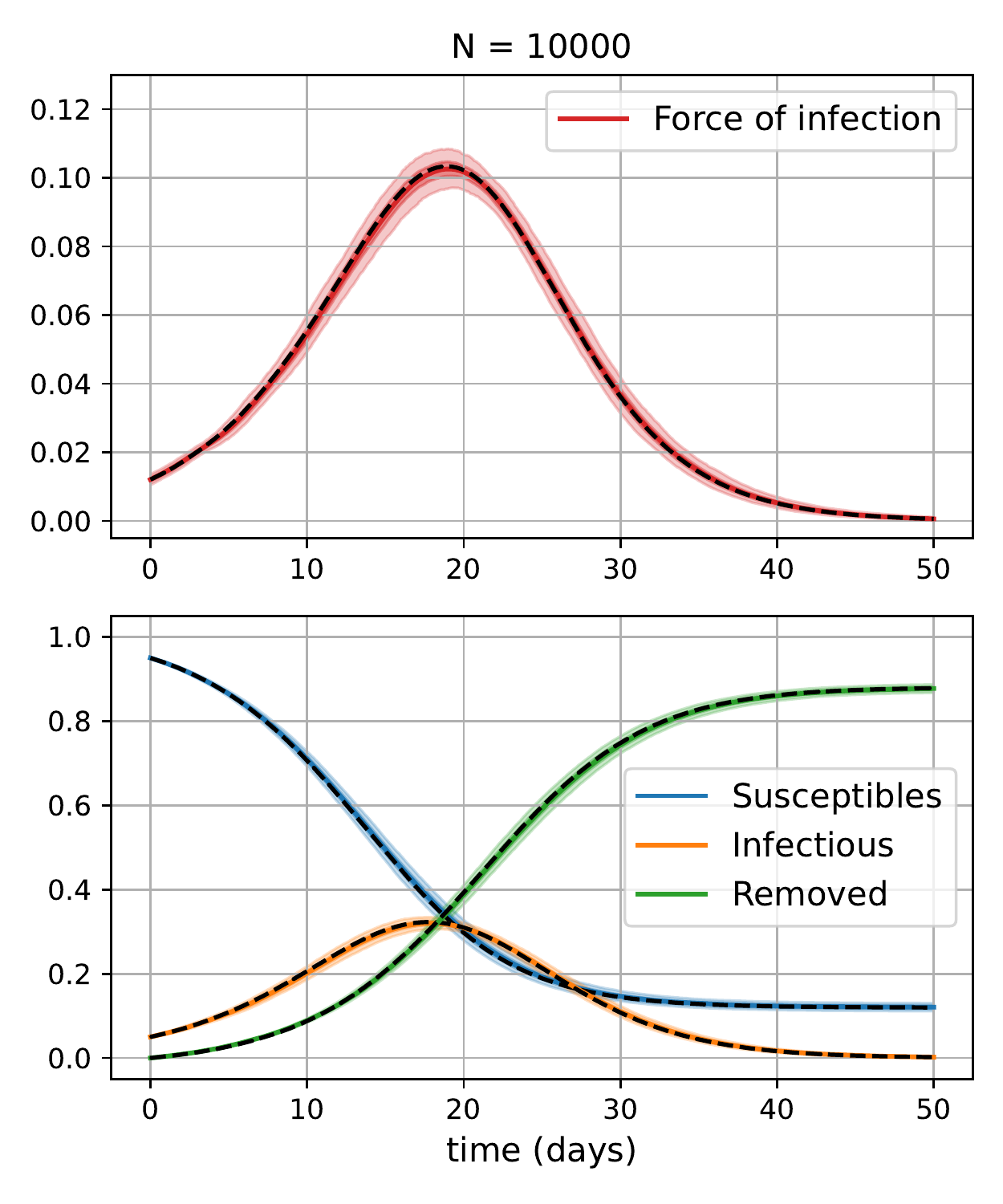}
	\caption{Numerical illustration of the FLLN obtained in Theorem~\ref{thm-FLLN} for the SEIR/SIR model (see below). Each graphic shows the mean of 1,000 independent simulations of the stochastic SEIR/SIR model (continuous lines) and the corresponding deterministic solution to \eqref{eqn-barS}-\eqref{eqn-barR} (black dashed lines), each started with $ \overline{I}^N(0) = \overline{I}(0) = 0.05 $.
	For each curve, the dark (resp. light) shaded areas around the curves represent the intervals containing 50\% (resp. 95\%) of the simulations.
	The two compartments E and I have been merged so as not to burden the graphic with another pair of curves (see below). The population size $ N = 10^3 $ on the left, $N= 10^4 $ on the right. The model and the distribution of $ (\zeta, \eta, \lambda) $ are as described in Subsection~\ref{subsec:covid} below, with $ p_R = 0.8 $, $ \alpha = 0.7 $.} \label{fig:FLLN}
\end{figure}

\paragraph{\bf The SEIR/SIR model}
Suppose now we do not want to follow the disease progression in the detail adopted so far. Rather, we merge the compartments E (exposed) and I  (infectious) into a single compartment I, where now I stands for infected, whether exposed or infectious. Doing this, we do not modify at all our model. 
Each newly infected individual belongs to the I compartment from the time of infection $\tau^N_i$ until the end of the infectious period $\tau^N_i+\zeta_i +\eta_i$, where again 
$\zeta_i+\eta_i=\inf\{t>0,\ \lambda_i(r)=0,\ \forall r\ge t\}$. 
Of course, between time $\tau^N_i$ and time $\tau^N_i+\zeta_i$, $\lambda_i(t)=0$ (recall that 
$\zeta_i=\inf\{t,\ \lambda_i(t)>0\}$), so that he/she is not infectious, but exposed. 
Likewise, each initially infected individual belongs to the I compartment from time 0 up to time $ \zeta_j^0 + \eta_j^0 $, where $ \zeta_j^0 +  \eta_j^0 = \inf \lbrace t \geq 0 : \lambda_j^0(r) = 0, \forall r \geq t \rbrace $.
Note that $ \zeta_j^0 = 0 $ if $ \lambda_j^0(0) > 0 $ (if the individual is already infectious at time 0).
As a result, \eqref{eqn-EN} and \eqref{eqn-IN} are replaced by
\begin{align} \label{eqn_IN_merged}
I^N(t) = \sum_{k=1}^{I^N(0)} {\bf 1}_{t < \zeta^0_k + \eta^0_k} + \sum_{i=1}^{A^N(t)} {\bf 1}_{t < \tau^N_i + \zeta_i + \eta_i},
\end{align}
and $ E^N(t) = 0 $ in all the other equations.
The force of infection is then
\begin{align} \label{eqn_mathfrakI_merged}
\mathfrak I^N(t) = \sum_{k=1}^{I^N(0)} \lambda^0_k(t) + \sum_{i=1}^{A^N(t)} \lambda_i(t-\tau_i^N).
\end{align}
We call this model the SEIR/SIR model, since it is an SIR model, but with 
I meaning ``infected'', and the state E is implicit, i.e. we do not exclude that individuals, when they become infected, are first exposed, then later infectious.
Define
\begin{align*}
F(t)&=\P(\zeta+\eta\le t),\quad\text{where }\zeta+\eta=\inf\{t>0,\ \lambda(r)=0,\ \forall r\ge t\},\\
F_0(t)&=\P(\zeta^0+\eta^0\le t),\quad\text{where }\zeta^0+\eta^0=\inf\{t>0,\ \lambda^0(r)=0,\ \forall r\ge t\}\,.
\end{align*}
With those notations, the deterministic LLN SEIR/SIR  model reads as follows.
 \begin{align}
 \bar{S}(t)&=1-\bar{I}(0)-\int_0^t\bar{S}(s)\bar{\mathfrak{I}}(s)ds\,, \label{seir-sirS}\\
\bar{\mathfrak{I}}(t)&=\bar{I}(0)\bar{\lambda}^{0,I}(t)+\int_0^t\bar{\lambda}(t-s)\bar{S}(s)\bar{\mathfrak{I}}(s)ds\,,
\label{seir-sirJ}\\
 \bar{I}(t) &= \bar{I}(0) F_{0}^c(t)   +\int_0^t F^c(t-s) \bar{S}(s)\bar{\mathfrak{I}}(s)ds\,,\label{seir-sirI}\\
\bar{R}(t)&=\ \bar{I}(0) F_{0}(t)   +\int_0^t F(t-s) \bar{S}(s)\bar{\mathfrak{I}}(s)ds\, .\label{seir-sirR}
 \end{align}

Now in the particular case where $\lambda^0(\cdot)$ and $\lambda(\cdot)$ are such that
$\zeta=\zeta^0=0$ a.s. (i.e., an infected individual is immediately infectious), there is no exposed period, then the above model is the generalized SIR model with varying infectivity.

Figure~\ref{fig:FLLN} illustrates the FLLN of Theorem~\ref{thm-FLLN} for the SEIR/SIR model, for two values of the population size ($ 10^3 $ and $ 10^4 $).
Each figure displays the mean of 1,000 independent simulations, the trajectory of the deterministic equations \eqref{eqn-barS}-\eqref{eqn-barR}, and the intervals containing 50\% and 95\% of the trajectories.
The details of the model and the distribution of $ (\zeta, \eta, \lambda) $ used in the simulations are described in Subsection~\ref{subsec:covid} below.
In each case, the mean of the simulations is almost superposed with the solution to the deterministic equations, and for $ N = 10^4 $, the envelopes are very concentrated around the means.
This is not surprising in view of the FCLT proved in \cite{PP-2020c}.
Indeed, this theorem implies that the trajectory of the (renormalised) stochastic process $ (\overline{S}^N(t), \overline{\mathfrak I}^N(t), \overline{I}^N(t), \overline{R}^N(t), t \geq 0) $ is (with high probability) at a distance of the order of $ N^{-1/2} $ from that of the deterministic limit.
The simulations obtained in Figure~\ref{fig:FLLN} confirm this, and the width of the 50\% and 95\% intervals are exactly proportional to $ N^{-1/2} $.

\begin{remark}
The above result generalizes both our SIR and our SEIR FLLN results in \cite{PP-2020}. 

The SIR model in \cite{PP-2020} is the particular case of the present result, where 
$\lambda(t)=\lambda{\bf1}_{t<\eta}$, $\eta$ being the random duration of the infectious period. In this case, $\bar{\lambda}(t)=\lambda F^c(t)$, if $F$ is the c.d.f. of $\eta$, and $F^c=1-F$. Note that in this case
$\bar{\mathfrak{I}}(t)=\lambda \bar{I}(t)$. Therefore, if we divide the $\bar{\mathfrak{I}}$ equation by $\lambda$, we find equation \eqref{eqn-barI}, 
which is also equation (2.4) in \cite{PP-2020}. If we assume that the law of $\eta$ is exponential, then we are in the case of the classical SIR model.

The SEIR model in \cite{PP-2020} corresponds to the situation where 
$\lambda(t)=\lambda{\bf1}_{\zeta\le t<\zeta+\eta}$, where $\zeta$ is the duration of the exposed period (the time when the individual is infected, but not yet infectious), and $\eta$ is as above, while
$\lambda^0(t)=\lambda{\bf1}_{\zeta^0\le t<\zeta^0+\eta^0}$. Then 
$\bar{\lambda}(t)=\lambda[\P(\zeta\le t)-\P(\zeta+\eta\le t)]=\lambda\Psi(t)$. If we divide the $\bar{\mathfrak{I}}$ equation by $\lambda$, we find equation \eqref{eqn-barI}, which is also (3.15) in \cite{PP-2020}. If moreover $\zeta$ and $\eta$ are independent exponential random variables, then we are reduced to the classical SEIR model.
\end{remark}

 \begin{remark} \label{rem-SIS}
 For the generalized SIS model, since $\bar{S}(t)=1-\bar{I}(t)$, it is clear that the epidemic dynamics in the FLLN is determined by the two--dimensional functions $\big(\bar{\mathfrak{I}}, \bar{I}\big)$ via the following integral equations:
  \begin{align*}
\bar{\mathfrak{I}}(t)&=\bar{I}(0)\bar{\lambda}^{0,I}(t)+\int_0^t\bar{\lambda}(t-s)(1-\bar{I}(s))\bar{\mathfrak{I}}(s)ds\,,\\
 \bar{I}(t) &= \bar{I}(0) F_{0,I}^c(t)   +\int_0^t F^c(t-s) (1-\bar{I}(s))\bar{\mathfrak{I}}(s)ds\,. 
 \end{align*}
 Recall that as shown in Theorem 2.3 of \cite{PP-2020}, in the SIS with general infectious periods, $\bar{\mathfrak{I}}(s) = \lambda \bar{I}(s)$, and the epidemic dynamics is determined by the one--dimensional integral equation for $\bar{I}$.
 
 For the generalized SIRS model,  the variables $(\zeta_i, \eta_i)$ in our setup represent the infectious and recovered/immune periods of newly infected individuals, and similarly the variables $(\zeta^0_j, \eta^0_j)$ represent the infectious and immune periods of initially infectious individuals. We assume that there is no initially immune individuals. Let $I^N, R^N$ be the processes counting infectious and recovered/immune individuals (corresponding to the notation $E^N$ and $I^N$ in the SEIR model).  
 Of course, instead of \eqref{eqn-lambda-def-1}, the infectivity function $\lambda(t)$ should be positive only in the infectious periods $[0,\zeta_i)$. Similarly, $\lambda^0_j(t)$ should be positive only over $[0, \zeta^0_j)$. 
 The definitions of the variables $(\zeta_i, \eta_i)$, $(\zeta_j^0, \eta^0_j)$  in \eqref{eqn-zeta-lambda} and \eqref{eqn-zeta-lambda-0} also need to be modified accordingly in the natural way.
 The distribution functions $G_0, F_{0,R}$ are for initially infectious and immune periods, and $G,F$ for newly infectious and immune periods, similarly for the notation $\Psi, \Psi_0, \Phi, \Phi_0$. 
Then the epidemic dynamics of the generalized SIRS model in the FLLN is determined by the three--dimensional functions $\big(\bar{\mathfrak{I}}, \bar{I}, \bar{R}\big)$ via the following integral equations:
\begin{align*} 
\bar{\mathfrak{I}}(t)&=\bar{I}(0)\bar{\lambda}^0(t)+\int_0^t\bar{\lambda}(t-s)\big(1- \bar{I}(s) - \bar{R}(s)\big)\bar{\mathfrak{I}}(s)ds\,,\\
 \bar{I}(t) &=\bar{I}(0)G_0^c(t)+\int_0^tG^c(t-s)\big(1- \bar{I}(s) - \bar{R}(s)\big) \bar{\mathfrak{I}}(s)ds\,, \\
 \bar{R}(t) &=  \bar{I}(0)\Psi_0(t)  +\int_0^t \Psi(t-s) \big(1- \bar{I}(s) - \bar{R}(s)\big) \bar{\mathfrak{I}}(s)ds\,.
\end{align*}
 Also recall that as shown in Theorem 3.3 of \cite{PP-2020}, in the SIRS model with general infectious and recovered periods, $\bar{\mathfrak{I}}(s) = \lambda \bar{I}(s)$, and the epidemic dynamics is determined by the two--dimensional integral equation for $\big(\bar{I}, \bar{R}\big)$.
 \end{remark}

\subsection{The early phase of the epidemic}\label{sec:early}

Theorem~\ref{thm-FLLN} shows that the deterministic system of equations \eqref{eqn-barS}-\eqref{eqn-barfrakI} accurately describes the evolution of the stochastic process defined in Subsection~\ref{subsec:model} when the initial number of infectious individuals is of the order of $ N $.
But epidemics typically start with only a handful of infectious individuals, and it takes some time before the epidemic enters the regime of Theorem~\ref{thm-FLLN}.
Exactly how long this takes depends on the population size $ N $ and on the growth rate of the epidemic.
To determine this growth rate, we study the behavior of the stochastic process when the initial number of infectious individuals is kept fixed as $ N \to \infty $.

In order to simplify the notations, we shall use the reduced model introduced in \eqref{eqn_IN_merged} and \eqref{eqn_mathfrakI_merged}, where exposed and infectious individuals are merged in a single infected compartment I.
We now suppose that $ I^N(0) = I(0) $ is a fixed random variable taking values in $ \lbrace 1, \ldots, N_0 \rbrace $ for some $ N_0 \geq 1 $, and we take $ N \geq N_0 $ throughout this section.

Let
\begin{align} \label{def_R0}
R_0 = \int_{0}^{\infty} \overline{\lambda}(t) dt,
\end{align}
and let $ \rho \in \R $ be the unique solution of
\begin{align} \label{def_rho}
\int_{0}^{\infty} \overline{\lambda}(t) e^{-\rho t} dt = 1.
\end{align}
The quantity $ R_0 $ is the well--known basic reproduction number, \textit{i.e.}, the average number of individuals infected by a typical infected individual in a large, fully susceptible population.
It is also well known that, if $ R_0 \leq 1 $, the total number of infections remains small as $ N \to \infty $, \textit{i.e.}, $ \limsup_{t \to \infty} A^N(t) $ converges in probability as $ N \to \infty $ to a random variable $ Z $ taking values in $ \N $, almost surely, see Corollary~1.2.6 in \cite{britton2018stochastic}.
If $ R_0 > 1 $, however, with positive probability, a major outbreak takes place, \textit{i.e.}, a positive fraction of the $ N $ individuals is infected at some point during the course of the epidemic.
The time needed in order to observe this major outbreak has been studied for Markovian epidemic models in \cite{barbour1975duration}.
More precisely, it has been shown that, starting from a fixed number of individuals, on the event that there is a major outbreak, the first time at which the proportion of infected individuals is at least $ \varepsilon > 0 $ is
\begin{align*}
\frac{1}{\rho} \log(N) + \mathcal{O}(1),
\end{align*}
as $ N \to \infty $, for any $ \varepsilon > 0 $ small enough, where $ \rho > 0 $ is given by \eqref{def_rho} (it can easily be seen that $ \rho > 0 $ if and only if $ R_0 > 1 $).
The aim of this section is to extend this result to our non--Markovian setting.

We thus let, for $ \varepsilon \in (0,1) $,
\begin{align*}
T^N_\varepsilon := \inf \lbrace t \geq 0 : A^N(t) \geq \varepsilon N \rbrace
\end{align*}
and, for any $ \alpha \in (0,1) $,
\begin{align*}
\mathcal{T}^N_\alpha := \inf \lbrace t \geq 0 : A^N(t) \geq N^\alpha \rbrace.
\end{align*}
Here and in what follows, we shall use $ X^N \Rightarrow X $ to denote the convergence in distribution of a sequence of random variables $ (X^N, N \geq 1) $ to a random variable $ X $ as $ N \to \infty $, \textit{i.e.}, $ X^N \Rightarrow X $ if and only if, for any continuous and bounded real-valued function $ \Phi $, $ \E \left[ \Phi(X^N) \right] \to \E \left[ \Phi(X) \right] $ as $ N \to \infty $.
We then have the following result, which we prove in Section~\ref{sec:early_phase}.

\begin{theorem} \label{thm:early_phase}
	Under Assumptions~\ref{AS-lambda} and \ref{AS-Lambda-Moment2},
	 for any $ \varepsilon > 0 $ such that $ \varepsilon < 1-\frac{1}{R_0} $, as $ N \to \infty $,
	\begin{align*}
	\frac{T^N_\varepsilon}{\log(N)} \Rightarrow \frac{1}{\rho} X,
	\end{align*}
	where $ X = +\infty $ with probability $ q $ and $ X = 1 $ otherwise, for some $ q \in (0,1) $.
	Moreover, for any $ \alpha \in (0,1) $,
	\begin{align*}
	\frac{\mathcal{T}^N_\alpha}{\log(N)} \Rightarrow \frac{\alpha}{\rho} X.
	\end{align*}
\end{theorem}

Theorem~\ref{thm:early_phase} essentially says that, on an event of probability close to $ 1-q $, $ t \mapsto A^N(t) $ grows approximately like (a constant times) $ t \mapsto e^{\rho t} $ until it becomes of the order of $ N $.
This exponential growth comes from the fact that, as long as $ \overline{S}^N(t) \approx 1 $, the infected individuals behave almost like a branching process (which in our case is non--Markovian, and is of the type studied in \cite{crump_general_1968,crump_general_1969}).
Since $ A^N(t) \approx e^{\rho t} $, this approximation is good as long as $ t \ll \frac{1}{\rho} \log(N) $, at which time the proportion of susceptible individuals is no longer close to one, and the branching process approximation breaks down.
We shall also see in the proof of Theorem~\ref{thm:early_phase} that $ q $ is equal to the extinction probability of this approximating branching process.

\begin{remark} \label{rk:eps_R0}
	The condition $ \varepsilon < 1-\frac{1}{R_0} $ comes from the fact that, as long as $ \overline{S}(t) < \frac{1}{R_0} $, each infected individual infects on average more than one susceptible individual.
	Hence the proportion of susceptible individuals needs to become lower than this threshold for the epidemic to die out (on the event that there is a major outbreak).
	As a result, $ A^N(t) $ has to exceed $ \varepsilon N $ for some time $ t < \infty $ for any $ \varepsilon < 1-\frac{1}{R_0} $.
\end{remark}

The fact that the number of infected individuals grows exponentially at rate $ \rho $ as long as the proportion of susceptible individuals stays close to one can also be seen from the deterministic equations by taking $ \overline{S}(t) = 1 $ in \eqref{seir-sirJ} (as well as \eqref{seir-sirI} and \eqref{seir-sirR}).
This substitution leads to the following (linear) system (recall that in this section $F$ is the distribution function of the r.v. $\zeta+\eta$): 
\begin{equation}\label{eqlin0}
\begin{split}
{\mathfrak{I}}(t) &= {I}(0)  \bar{\lambda}^0(t) + \int_0^t \bar{\lambda}(t-s) {\mathfrak{I}}(s)ds \,,\\
{I}(t) &= {I}(0) F_{0}^c(t)   +\int_0^t F^c(t-s) {\mathfrak{I}}(s)ds\,,\\
{R}(t)&= R(0)+  {I}(0) F_{0}(t)   +\int_0^t F(t-s) {\mathfrak{I}}(s)ds\, . 
\end{split}
\end{equation}
We prove the following in Section~\ref{sec:early_phase}.

\begin{theorem} \label{thm:early_phase_deterministic}
Assume that Assumption \ref{AS-lambda} holds true.
	For $ \rho \in \R $, suppose that $ \E \big[ e^{-\rho (\zeta + \eta)} \big] < \infty $ and define
	\begin{align} \label{def_i_r}
	\bm{i} := \int_{0}^{\infty} F^c(s) \rho e^{-\rho s} ds, && \bm{r} := 1 - \bm{i},
	\end{align}
	and
	\begin{align*}
	\overline{\lambda}_\rho(t) := \frac{\int_{0}^{\infty} \overline{\lambda}(t+s) e^{-\rho s} ds}{\int_{0}^{\infty} F^c(s) e^{-\rho s} ds}, && F_\rho^c(t) := \frac{\int_{0}^{\infty} F^c(t+s) e^{-\rho s} ds}{\int_{0}^{\infty} F^c(s) e^{-\rho s} ds}.
	\end{align*}
	Suppose first that $ R_0 > 1 $ and that $ \rho > 0 $ is the solution to \eqref{def_rho}. 
	Then, if $ \overline{\lambda}^0 = \overline{\lambda}_\rho $ and $ F_0 = F_\rho $, the linear system \eqref{eqlin0} admits the following solution
	\begin{align} \label{det_solution_exponential}
	\mathfrak{I}(t)=\rho\, e^{\rho t}, \quad I(t)=\bm{i}\, e^{\rho t}, \quad R(t) = \bm{r}\, e^{\rho t}\, \quad t \geq 0.
	\end{align}
	If, however, $ R_0 < 1 $ and $ \rho < 0 $ (still satisfying \eqref{def_rho}), then the linear system \eqref{eqlin0} (with $ \overline{\lambda}^0 = \overline{\lambda}_\rho $ and $ F_0 = F_\rho $) admits the following solution
	\begin{align*}
	\mathfrak I (t) = -\rho e^{\rho t}, \quad I(t) = -\bm{i} e^{\rho t}, \quad R(t) = R(0) + \bm{r} (1-e^{\rho t}), \quad t \geq 0.
	\end{align*}
\end{theorem}

The deterministic system \eqref{eqlin0} can be thought of as an approximation of the expectation of the stochastic process $ (\mathfrak{I}^N(t), I^N(t), R^N(t)) $ when $ \overline{S}^N(t) \approx 1 $.
Note that if we take the exponentially growing solution \eqref{det_solution_exponential} and if we set
\begin{align*}
	A(t) := I(t) + R(t) - (I(0) + R(0))
\end{align*}
(which corresponds to the number of newly infected individuals up to time $ t $), then, since $ \bm{i}+\bm{r} = 1 $, $ A(t) = e^{\rho t}-1 $ and
\begin{align} \label{expected_A}
	A\left( \frac{\alpha}{\rho} \log(N) \right) = N^\alpha - 1 \sim N^\alpha.
\end{align}
Hence Theorems~\ref{thm:early_phase} and \ref{thm:early_phase_deterministic} show that the stochastic model and the linear deterministic system \eqref{eqlin0} have the same asymptotical behavior, on the event that there is a major outbreak, for times of the form $ \frac{\alpha}{\rho} \log(N) $, $ \alpha \in (0,1)$.
This is further illustrated in Figure~\ref{fig:exp_growth}, which displays the mean of a subset 1,000 independent copies of $ t \mapsto I(0) + A^N(t) $ for which the epidemic didn't go extinct at the beginning.
We see on the figure that, after an initial stochastic phase, whose duration may vary between different realizations, the cumulative number of infected individuals indeed grows at the expected rate $ \rho $.
We also see that the slope of $t\mapsto I(0) + A^N(t)$ starts to decline when $A^N(t)$ exceeds $N/10$ (hence when 
$\bar{S}^N(t)$ becomes less than $0.9$), which is to be expected from the deterministic model.

\begin{figure}[ht]
	\centering
	\includegraphics[width=0.7\textwidth]{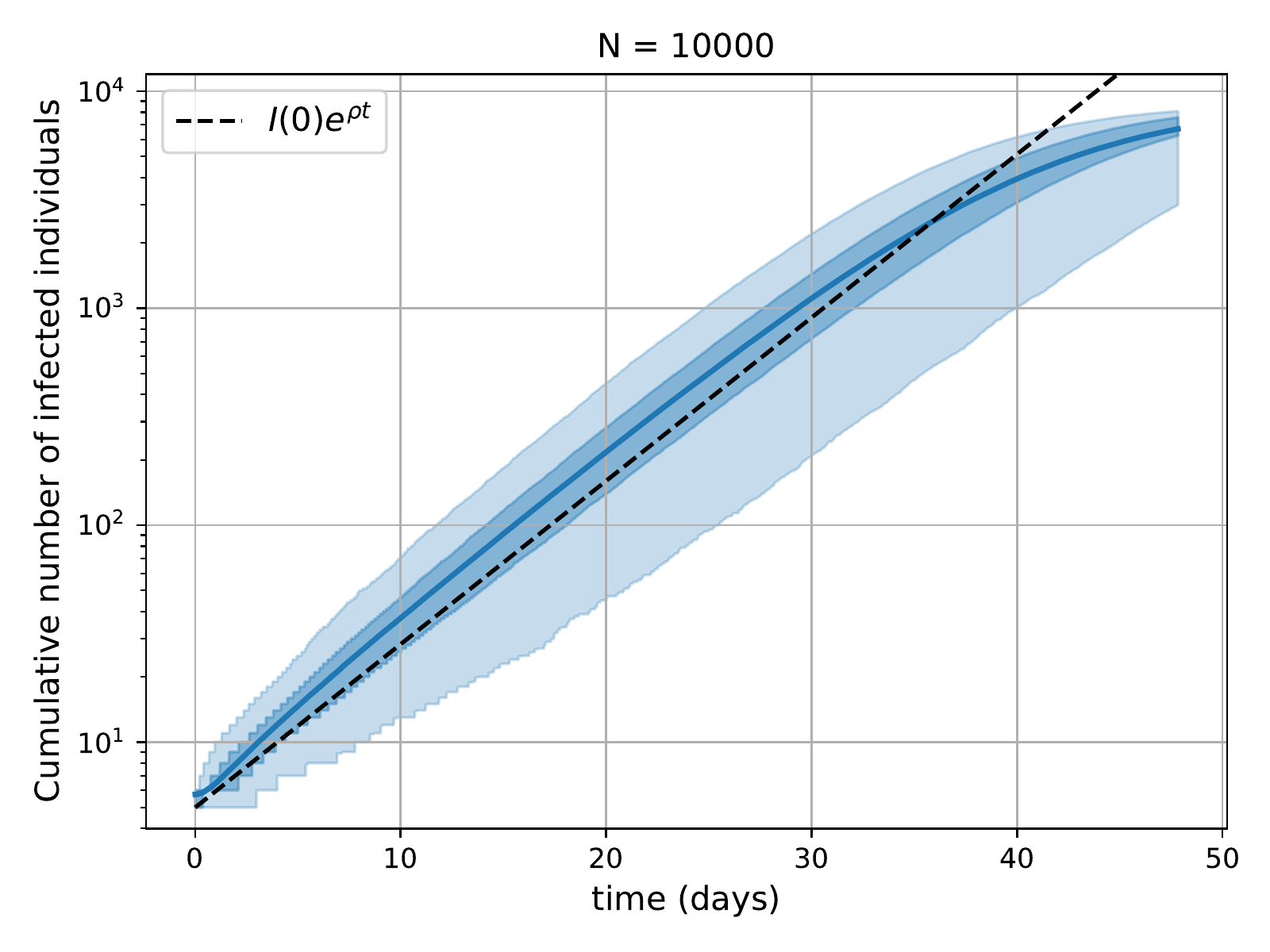}
	\caption{Exponential growth of the cumulative number of infected individuals $ t \mapsto I(0) + A^N(t) $ in the stochastic model. The figure shows the mean (blue line), 50\% envelope (dark blue region) and 95\% envelope (light blue region) of the subset of 1,000 independent simulations for which the epidemic did not go extinct at the beginning. Each simulation was started with $ I(0) = 5 $ infectious individuals and a population size of $ N = 10^4 $. The dashed black line shows the expected exponential growth during this early phase $ t \mapsto I(0) e^{\rho t} $ (the factor $ I(0) $ arises from the branching property). The mean of the sample is slightly above the dashed line, owing to the bias resulting from the fact that only trajectories leading to a major outbreak were kept.} \label{fig:exp_growth}
\end{figure}

In the case of Markovian (SIR) epidemic models, Theorem~2 of \cite{barbour1975duration} states that the full duration of the epidemic (\textit{i.e.}, the time to extinction of the I population) $ T_N $, when starting from a single infected individual, satisfies
\begin{align*}
	\P\left( T_N - a \log(N) - c \geq x \right) \to (1-q)\, \P\left( W \geq x \right), \quad N \to \infty,
\end{align*}
for some constants $ a > 0 $ and $ c \in \R $, where $ W $ is a linear combination of two independent Gumbel random variables.
Moreover, $ a = \frac{1}{\rho} + \frac{1}{\rho'} $, where $ \rho $ is the same as in Theorem~\ref{thm:early_phase} and $ \rho' $ is the rate of decay of the number of infected individuals during the final stage of the epidemic.
In addition, Theorem~1.1 in \cite{barbour_escape_2015} shows that the  stochastic process can be coupled with a branching process so that the two follow the same trajectory up to the time $ \min(T^N_0, \mathcal{T}^N_\alpha) $, for $ \alpha = 7/12 $, except on an event of asymptotical negligible probability.
Moreover, Theorem~1.1 in \cite{barbour_escape_2015} also says that, for times of the form $ \mathcal{T}^N_\alpha + t $, for $ 0 \leq t \leq \frac{1-\alpha}{\rho} \log(N) + T $, the trajectory of the stochastic process is, with high probability, at most at distance $ k N^{-\gamma} $ of the trajectory of a solution of the deterministic (non-linear) equations \eqref{seir-sirS}--\eqref{seir-sirR}, whose initial condition is of the form
\begin{align*}
	\overline{S}(0) = 1-\frac{I(0)}{N}, && \overline{I}(0) = \frac{I(0)}{N},
\end{align*}
up to a time shift which stays of the order of 1 as $ N \to \infty $, and which accounts for the stochastic fluctuations when the number of infected individuals is small.
We expect that a similar result holds in our non-Markovian setting, but proving this would require a careful comparison of the stochastic model with the deterministic model started from an $ \mathcal{O}(1/N) $ initial proportion of infected individuals over timescales of the order of $ \log(N) $, and this would go beyond the scope of this paper.

The second part of the statement (when $ R_0 < 1 $) describes what takes place when the daily number of new infections is decreasing, either because a large fraction of the population has been infected (or vaccinated) or because effective containment measures have been put into place (\textit{e.g.}, a strict lockdown).
In the former case, $ \overline{S}(t) $ is not close to one, and $ \overline{\lambda} $ should be replaced by $ \overline{S}(t) \overline{\lambda} $ in order to determine $ \rho $ and $ \overline{\lambda}_\rho $ (assuming that $ \overline{S}(t) $ varies slowly at this point).

Note that if we replace $ I(0) $, $ R(0) $, $ \overline{\lambda}^0 $ and $ F_0 $ by their values in Theorem~\ref{thm:early_phase_deterministic}, and if we set, for $ t < 0 $,
\begin{align*}
\mathfrak I (t) = \rho e^{\rho t}, && I(t) = \bm{i} e^{\rho t}, && R(t) = \bm{r} e^{\rho t},
\end{align*}
then we have
\begin{align*}
\mathfrak I(t) = \int_{-\infty}^{t} \overline{\lambda}(t-s) \mathfrak I(s) ds, &&
I(t) = \int_{-\infty}^{t} F^c(t-s) \mathfrak I(s) ds, &&
R(t) = \int_{-\infty}^{t} F(t-s) \mathfrak I(s) ds.
\end{align*}
Hence \eqref{eqlin0} can also be interpreted as the (expected) behavior of an epidemic which has started from an infinitesimal number of infected individuals very far back in the past.
Incidentally, substituting $ \mathfrak I(t) = \rho e^{\rho t} $ in the first equation yields exactly \eqref{def_rho}.

\subsection{Estimating the basic reproduction number for an ongoing an epidemic} \label{sec:estimating_contact_rate}

The function $ \overline{\lambda} $ (as well as $ F $) depends on many factors.
Some of these factors are related to the evolution of the pathogen inside an infected individual's organism, and how easily it can be transmitted to neighboring individuals, and some of these factors depend on the intensity of social contacts in the population, in particular on physical contacts between individuals when they meet (hand shaking, kiss, hug, or none of those).
This function is affected by changes in social contacts and collective behaviors, including public policies aimed at mitigating the effects of the epidemic, and the use of face masks.
For example,  during the Covid-19 pandemic, many countries implemented strict lockdowns in order to curb the spread of the disease, which drastically reduced the rate of infectious contacts and significantly affected the growth rate of the number of newly infected individuals.
In order to estimate the impact of such policies in terms of the dynamics of the epidemic, we thus need to be able to gather some information on the contact rate $ \overline{\lambda} $ from the available data at some given time.

Let us suppose that $ \overline{\lambda} $ is only known up to a constant factor $ \mu > 0 $, \textit{i.e.}, 
\begin{align*}
\overline{\lambda}(t) = \mu\, \overline{g}(t), \quad t \geq 0,
\end{align*}
where $ \mu $ is unknown but $ \overline{g} $ is known (for example from medical data on viral shedding).
We can then estimate $ \mu $ (and $ R_0 $) from the growth rate $ \rho $, which can be measured easily at the beginning of the epidemic ($ \rho = \log(2) / d $, where $ d $ is the doubling time of the daily number of newly infected individuals), using the relation \eqref{def_rho}.
The following is thus a corollary of Theorem~\ref{thm:early_phase}.

\begin{coro} \label{coro-R0}
	Let $ \rho $ be the growth rate of the number of infected individuals.
	Then
	\begin{align*}
	\mu = \left( \int_{0}^{\infty} \overline{g}(t) e^{-\rho t} ds \right)^{-1},
	\end{align*}
	and the basic reproduction number $ R_0 $ is given by
	\begin{align} \label{eqn-R0}
	R_0 = \frac{\int_{0}^{\infty} \overline{g}(t) dt}{\int_{0}^{\infty} \overline{g}(t) e^{-\rho t} dt}.
	\end{align}
\end{coro}

In the literature, $ (\int_{0}^{\infty} \overline{g}(t) dt )^{-1}\overline{g}(t) $ is called the generation interval distribution (it is the distribution of the interval between the time at which an individual is infected and the time at which its ``children'' are infected).
The relation \eqref{eqn-R0} is thus (2.7) in \cite{wallinga_how_2007}. Note that $R_0$ is the mean multiplicative factor of the epidemic from one generation to the next, while $\rho$ is a growth factor in continuous time. 

Note that, by the second part of Theorem~\ref{thm:early_phase_deterministic}, \eqref{eqn-R0} remains valid on any interval during which $ \overline{S}(t) \approx \overline{S}(t_0) $ remains approximately constant (but not necessarily close to 1), even when $ \rho \leq 0 $.
In that case, one should add a factor $ \overline{S}(t_0) $ in front of $ \mathfrak I(s) $ on the right hand sides of \eqref{eqlin0}, and we obtain
\begin{align*}
	\mu\, \overline{S}(t_0) \int_0^\infty\bar g(s)e^{-\rho_e s} ds = 1.
\end{align*}
Hence if we define the \textit{effective} reproduction number $ R_e $ by $ R_e := \overline{S}(t_0) \int_{0}^{\infty} \overline{\lambda}(t) dt $ (\textit{i.e.}, the average number of secondary infections when $ \overline{S}(t) = \overline{S}(t_0) $), we have
\begin{align*}
	R_e = \overline{S}(t_0) R_0 = \frac{\int_{0}^{\infty} \overline{g}(s) ds}{\int_{0}^{\infty} \overline{g}(s) e^{-\rho_e s} ds}\,.
\end{align*}

\begin{remark}
	Note that the exponent $\rho$ is a quantity which is deduced from the observation of the epidemic (it is closely related to the ``doubling time'' of the number of cases).
	The above results give us $\mu$ and $R_0$ in terms of $\rho$ and the function $\bar{g}(t)$.
	If $\lambda(t)$ is deterministic, so are $ g(t) $ and $\eta$ and thus
	$$
	R_0 = \frac{\int_\zeta^{\zeta+\eta}g(s)ds}{ \int_\zeta^{\zeta+\eta} g(s) e^{-\rho s}ds}.   
	$$
	If, in addition, $\bar g(t) \equiv g>0$ for $ \zeta \leq t < \zeta + \eta $, then this simplifies to the well--known result
	$$
	R_0= \frac{\rho\eta} {e^{-\rho\zeta}(1- e^{-\rho \eta})}. 
	$$
\end{remark}

\begin{remark}
	Theorem \ref{thm:early_phase_deterministic} and its Corollary generalize Proposition 2 and Corollary 3 in \cite{FPP2020a}, in the case $\lambda(t)= \lambda{\bf1}_{\zeta\le t<\zeta+\eta}$ for some constant $\lambda>0$,  and the pair $(\zeta,\eta)$ is an arbitrary $\R_+^2$--valued random vector.
	In that case, our formula for $R_0$ reduces to \[ R_0=\frac{\rho\, \E[\eta]}{\E[e^{-\rho\zeta}(1-e^{-\rho\eta})]}\,.\]
	In the particular case where $\zeta$ and $\eta$ are independent exponential random variables, with parameters $\nu$ and $\gamma$, 
	the above formula becomes
	\[ R_0=\left(1+\frac{\rho}{\nu}\right)\left(1+\frac{\rho}{\gamma}\right)\,.\]
	From this we deduce the formula in the classical SIR case by choosing $\nu=+\infty$, i.e., 
	$$R_0=1+\frac{\rho}{\gamma}.$$
\end{remark}

\subsection{Application to the Covid--19 epidemic} \label{subsec:covid}
We now want to explain how the type of model described in this paper can be used to model the Covid--19 epidemic. As we have seen, the increase in realism with respect to the classical ``Markovian'' models (where the infectivity is constant and fixed across the population, and the Exposed and Infectious periods follow an exponential distribution) is paid by replacing a system of ODEs by  a system of Volterra integral equations. However, we have a small benefit in that the flexibility induced by the fact that the law of $\lambda$ is arbitrary allows us to reduce the number of compartments in the model, so that we can replace a system of ODEs 
by a system of Volterra type equations of smaller dimension. 

\begin{figure}[ht]
	\centering
	\includegraphics[width=.8\linewidth]{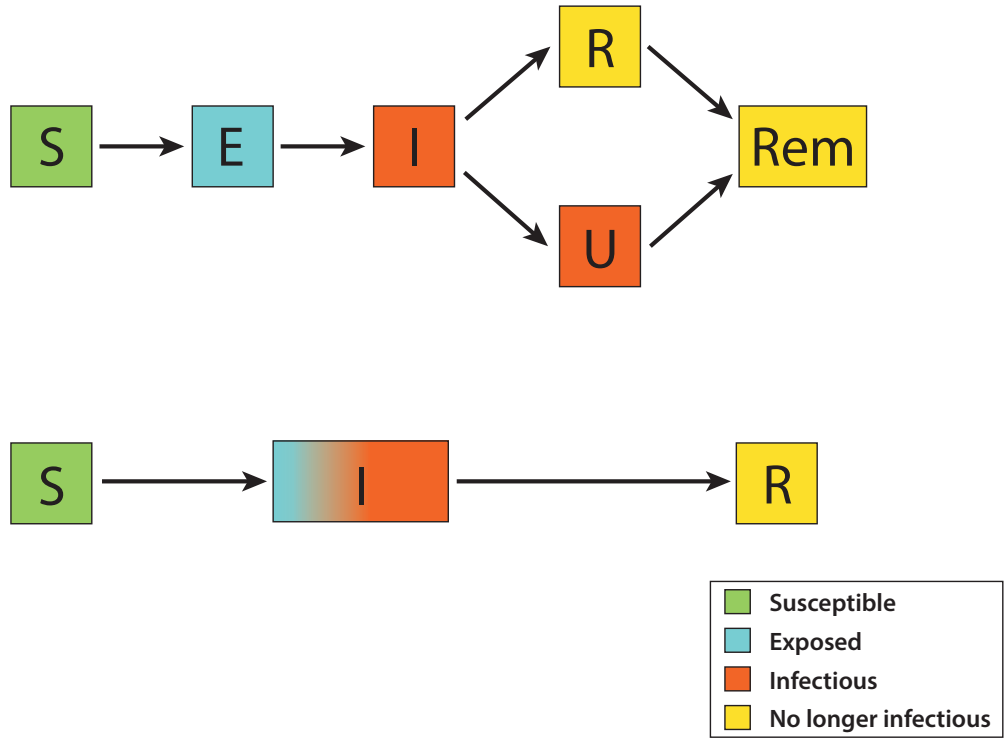}
	\caption{Flow chart of the SEIRU model of \cite{LMSW} and of our SIR model. We are able to replace the six compartments of the SEIRU model with only three compartments by using the equations described in Theorem~\ref{thm-FLLN}.}\label{fig:SEIR}
\end{figure} 

To be more specific, let us describe the SEIRU model of \cite{LMSW}. An individual who is infected is first ``Exposed'' E, then ``Infectious" I. Soon after, the infectious individual either develops significant symptoms, and then will be soon ``Reported'' R, and isolated so that he/she does not infect any more; while the alternative is that this infectious individual is asymptomatic: he/she develops no or very mild symptoms, so remains ``Unreported" U, and continues to infect susceptible individuals for a longer period. 
Both unreported and reported cases eventually enter the ``Removed" (Rem.) compartment.
In this model, there are 6 compartments: S like susceptible, E like exposed, I like infectious, R like reported, U like unreported, and Rem like removed.

Our approach allows us to have a more realistic version of this model with only 3 compartments (see Figure \ref{fig:SEIR}): S like susceptible, I like infected (first exposed, then infectious), R like removed (which includes the Reported individuals, since they do not infect any more, and will recover soon or later).  As already explained, we do not need to distinguish between the exposed and infectious, since the function $\lambda$ is allowed to remain equal to zero during a certain time interval starting from the time of infection. More importantly, since the law of $\lambda$ is allowed to be bimodal, we can accommodate in the same compartment I individuals who remain infectious for a short duration of time, and others who will remain infectious much longer (but probably with a lower infectivity). Moreover, since we know, see \cite{he2020temporal}, that the infectivity decreases after a maximum which in the case of symptomatic individuals, seems  to take place shortly before symptom onset, our varying infectivity model allows us to use a model corresponding to what the medical science tells us about this illness. Note that our version of the SEIRU model from \cite{LMSW} is the same as the one which we have already used in \cite{FPP2020a} (except that there we had to distinguish the E and the I compartments). However, the main novelty here is that the infectivity decreases after a maximum near the beginning of the infectious period.

\begin{figure}[ht]
	\centering
	\includegraphics[width=0.5\textwidth]{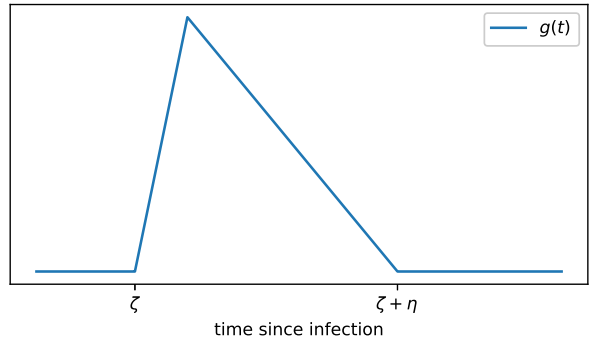}
	\caption{Profile of the function $g(t)$ used in our computation of $ R_0 $ as a function of $ \zeta $ and $ \eta $. The function increases linearly (up to a value 1 or $ \alpha $ depending on whether the individual is reported or unreported) on the interval $ [\zeta, \zeta + \eta/5] $ and then decreases linearly on $ [\zeta + \eta/5, \zeta + \eta] $.} \label{fig:g_t_profile}
\end{figure}

More precisely, we consider that $ t \mapsto g(t) $ increases linearly on the interval $ [\zeta, \zeta + \eta / 5] $, from 0 to 1 for reported individuals, and from 0 to $ \alpha $ for unreported individuals, and that it then decreases linearly to 0 on the interval $ [\zeta + \eta / 5, \zeta + \eta] $, as shown on Figure~\ref{fig:g_t_profile}.
We then take $ (X_1, X_2) $ a pair of independent Beta random variables with parameters (2, 2) and we assume that
\begin{align*}
	\zeta = 2 + 2 X_1, && \eta = \begin{cases}
	3 + X_2 & \text{ for reported individuals,} \\
	8 + 4 X_2 & \text{ for unreported individuals.}
	\end{cases}
\end{align*}
This joint law of $ (\zeta, \eta) $ is the one that was used in \cite{FPP2020a}
to study the Covid--19 epidemic in France (where the infectivity was assumed to be constant and uniform among individuals in this work), and these values are compatible with the results described in \cite{he2020temporal}.

Numerical results are presented in Figure~\ref{fig:R0} for three growth rates (0.277, -0.06, 0.032) which are derived from the doubling/halving times of the number of hospital deaths during the first wave (doubling time of 2.5 days), the first lockdown (halving time of 11.6 days) and the second wave (doubling time of 21.4 days) of the Covid--19 epidemic in France \cite{FPP2020a}. 
We note that, when $ \rho > 0 $ (resp. when $ \rho < 0 $), $R_0$ is increasing (resp. decreasing) with the proportion of unreported individuals and with $\alpha$. 
We also note that with the same durations of the exposed and infectious periods, but with $\lambda(t)$ constant, $R_0$ would be larger, which is not surprising, since in the present model the decrease of $\bar{\lambda}(t)$ reduces the effect of the factor $e^{-\rho t}$ in the integrals in the denominator, which makes $R_0>1$ for $\rho>0$.

\begin{figure}[h!]
	\centering
	\includegraphics[width=.5\textwidth]{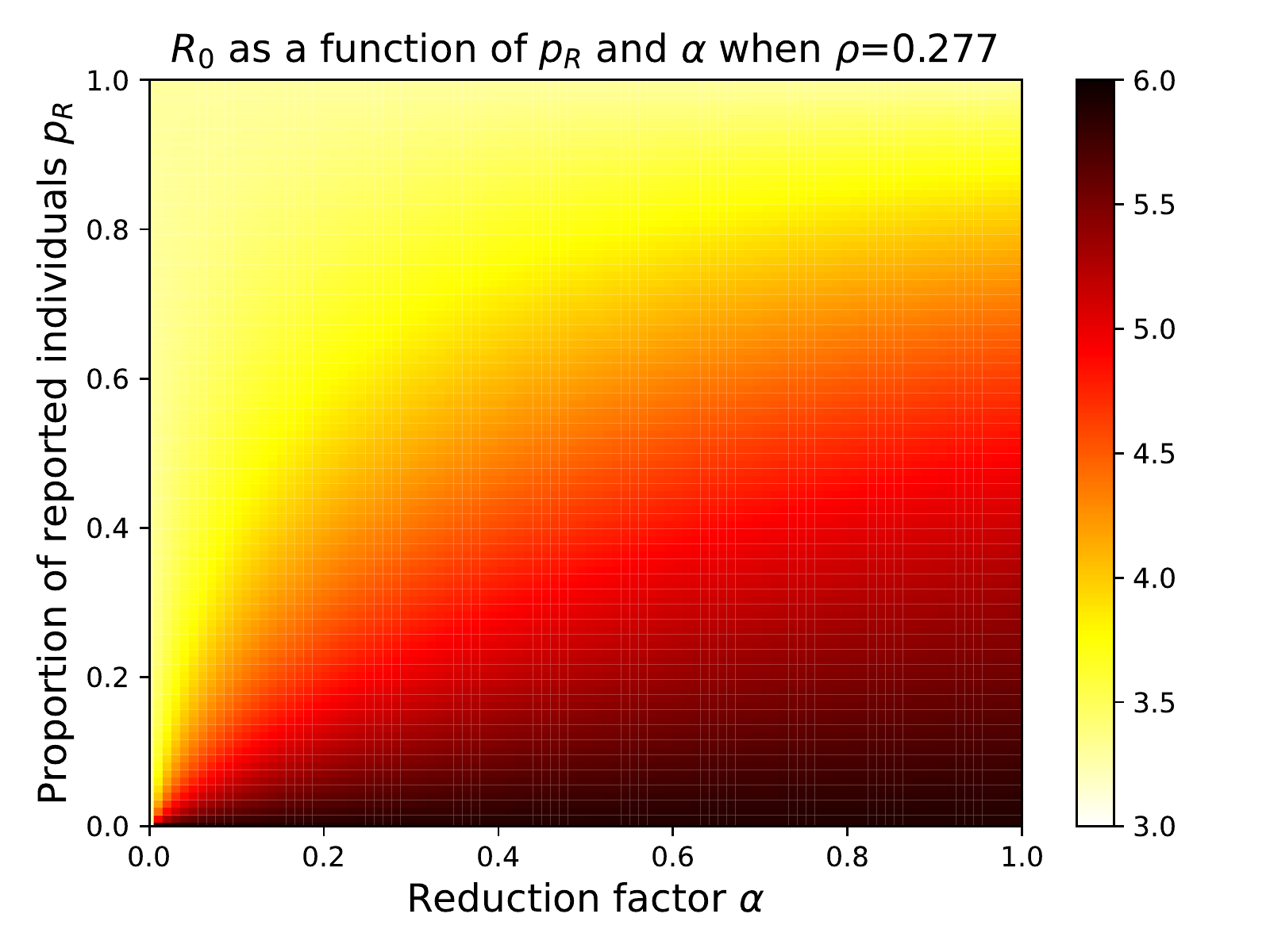}%
	\includegraphics[width=.5\textwidth]{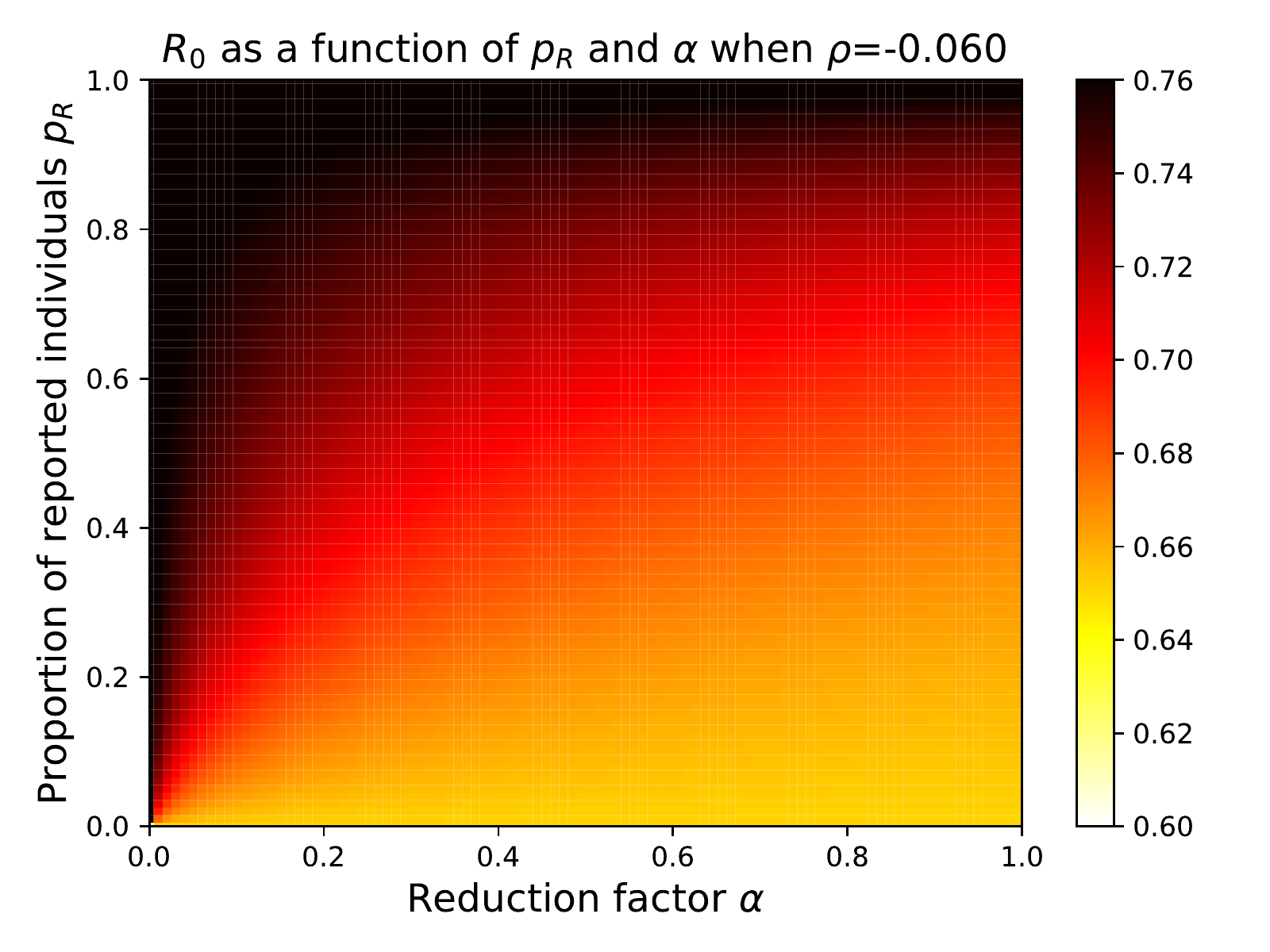}\par
	\includegraphics[width=.5\textwidth]{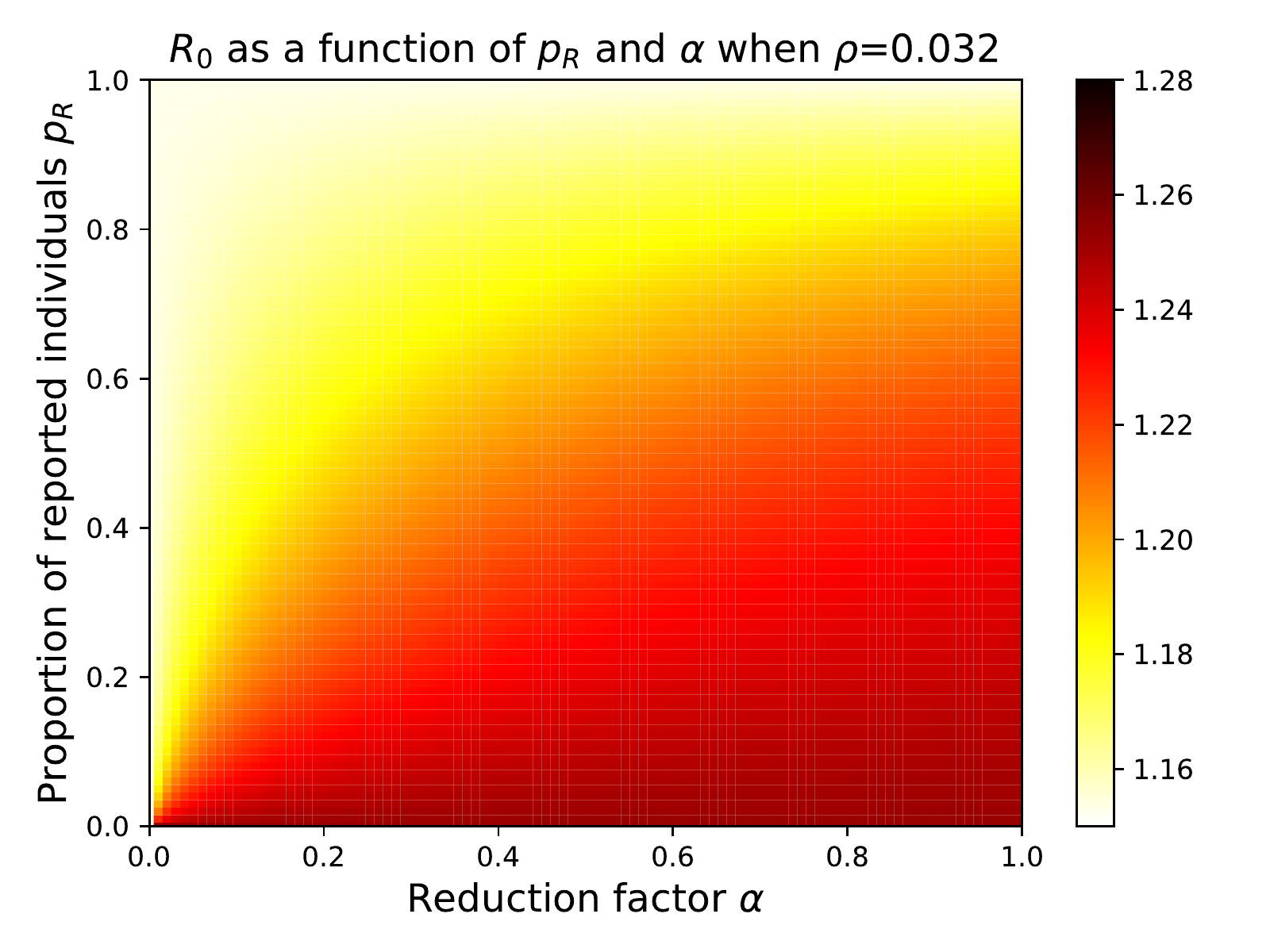}\par
	\caption{Heatmap of the value of $ R_0 $ for three growth rates: 0.277 (doubling time of 2.5 days), -0.06 (halving time of 11.6 days) and 0.032 (doubling time of 21.4 days), corresponding to three phases of the Covid--19 epidemic in France. In each graphic, the horizontal coordinate is the factor $ \alpha $ (which is the relative infectivity of unreported individuals compared to reported individuals), and the vertical coordinate is the proportion of reported individuals $ p_R $. Note that the range of values varies significantly with the growth rate $ \rho $ (from 3 up to 6 in the leftmost graphic, from 0.6 to 0.76 in the middle one and from 1.15 up to 1.28 in the rightmost graphic).} \label{fig:R0}
\end{figure}

\section{The early phase of the epidemic} \label{sec:early_phase}

The aim of this section is to prove Theorem~\ref{thm:early_phase} and Theorem~\ref{thm:early_phase_deterministic}.
In particular, we assume in this section that $\E \left[ \left( \int_0^\infty \lambda(t) dt \right)^2 \right] < \infty$ and that Assumption~\ref{AS-lambda} is satisfied.
The first step is to couple the stochastic process $ (A^N(t), \mathfrak I^N(t), t \geq 0) $ with two branching processes such that, at least up to some stopping time, the stochastic process $ A^N $ stays between the two branching processes.
To do this, we redefine the model of Subsection~\ref{subsec:model} in the following way.
Let $ (\lambda^0_k(\cdot), k \geq 1) $ be as before and let $ Q $ be a PRM on $ \R_+^2 \times D $ with intensity $ ds \otimes du \otimes P(d\lambda) $, where $ P $ is the probability distribution of $ \lambda(\cdot)$.
We then set
\begin{align*}
&\mathfrak I^N(t) := \sum_{k=1}^{I(0)} \lambda^0_k(t) + \int_{0}^{t} \int_{0}^{\infty}\int_D \lambda(t-s) {\bf 1}_{u \leq \Upsilon^N(s^-)} Q(ds, du, d\lambda), \\
&A^N(t) := \int_{0}^{t} \int_{0}^{\infty} \int_D {\bf 1}_{u \leq \Upsilon^N(s^-)} Q(ds, du, d\lambda),
\end{align*}
with $ \Upsilon^N(t) = \frac{S^N(t)}{N} \mathfrak I^N(t) $ and $ S^N(t) = N - I(0) - A^N(t) $ as before.
Then, for $ \varepsilon \in [0,1) $, we define
\begin{align*}
&\mathfrak{I}_\varepsilon(t) := \sum_{k=1}^{I(0)} \lambda^0_k(t) + \int_{0}^{t} \int_{0}^{\infty}\int_D \lambda(t-s) {\bf 1}_{u \leq (1-\varepsilon) \mathfrak{I}_\varepsilon(s^-)} Q(ds, du, d\lambda), \\
&A_\varepsilon(t) := \int_{0}^{t} \int_{0}^{\infty}\int_D {\bf 1}_{u \leq (1-\varepsilon) \mathfrak{I}_\varepsilon(s^-)} Q(ds, du, d\lambda).
\end{align*}
Recall that, for any $ \varepsilon \in [0,1) $,
\begin{align*}
T^N_\varepsilon = \inf \lbrace t \geq 0 : A^N(t) \geq \varepsilon N \rbrace.
\end{align*}

\begin{lemma} \label{lemma:coupling}
	For each $ N \geq N_0 $, the process $ (\mathfrak I^N(t), S^N(t), A^N(t), t \geq 0) $ has the same distribution as the one defined in Subsection~\ref{subsec:model}.
	Moreover,
	\begin{align} \label{upper_bound_branching}
	\forall t \geq 0, \quad \mathfrak I^N(t) \leq \mathfrak{I}_0(t), \quad A^N(t) \leq A_0(t),
	\end{align}
	and, for all $ 0 < \varepsilon < \varepsilon' $, for $ N \geq \frac{N_0+1}{\varepsilon'-\varepsilon} $, almost surely,
	\begin{align} \label{lower_bound_branching}
	\forall t \leq T^N_\varepsilon, \quad \mathfrak I^N(t) \geq \mathfrak{I}_{\varepsilon'}(t), \quad A^N(t) \geq A_{\varepsilon'}(t).
	\end{align}
\end{lemma}

We note that, even though the distribution of $ (\mathfrak I^N, A^N, S^N) $ is the same as in Subsection~\ref{subsec:model}, this construction yields a different coupling between $ (\mathfrak I^{N_1}, A^{N_1}, S^{N_1}) $ and $ (\mathfrak I^{N_2}, A^{N_2}, S^{N_2}) $ for $ N_1 \neq N_2 $.
\begin{proof}
	The fact that this new construction does not change the law of the process $ (\mathfrak I^N, S^N, A^N) $ is straightforward.
	For the second part of the statement,
		 let
		\begin{align*}
		\tau_0 := \inf \lbrace t \geq 0 : \mathfrak I^N(t) > \mathfrak{I}_0(t) \rbrace.
		\end{align*}
		By construction, if $ \tau_0 < \infty $, there exist $ s \leq \tau_0 $ and $ u > 0 $ such that
		\begin{align*}
			Q\left( \lbrace s \rbrace \times \lbrace u \rbrace \times D \right) = 1
		\end{align*}
		and 
		\begin{align*}
			\mathfrak I_0(s^-) < u \leq \Upsilon^N(s^-).
		\end{align*}
		Since $ \Upsilon^N(t) \leq \mathfrak I^N(t) $, this implies $ \mathfrak I_0(s^-) < \mathfrak I^N(s^-) $ for some $ s \leq \tau_0 $.
		This contradicts the definition of $ \tau_0 $, hence $ \tau_0 = + \infty $ and $ \mathfrak I^N(t) \leq \mathfrak I_0(t) $ for all $ t \geq 0 $.
		By the definition of $ A^N $ and $ A_0 $, this also implies $ A^N(t) \leq A_0(t) $ for all $ t \geq 0 $.
	
	For the lower bound \eqref{lower_bound_branching}, we note that, for $ t \leq T^N_\varepsilon $,
	\begin{align*}
	\Upsilon^N(t) &= \left( 1-\frac{I(0) + A^N(t)}{N} \right) \mathfrak I^N(t) \\
	&\geq \left( 1- \frac{N_0+1}{N} - \varepsilon \right) \mathfrak I^N(t) \\
	&\geq (1-\varepsilon') \mathfrak I^N(t),
	\end{align*}
	for $ N \geq (N_0+1)/(\varepsilon'-\varepsilon) $.
	The lower bound then follows by a similar argument as above.
\end{proof}

We note that the process $ A_\varepsilon(\cdot) $ does not depend on $ N $, and that it is a branching process which belongs to the class of processes studied in \cite{crump_general_1968,crump_general_1969}.
The following result is then Theorem~3.2 in \cite{crump_general_1969}.

\begin{lemma} \label{lemma:exp_branching}
	Under Assumptions~\ref{AS-lambda} and \ref{AS-Lambda-Moment2},
	for each $ \varepsilon \in [0,1) $, there exists a random variable $ W_\varepsilon \geq 0 $ such that
	\begin{align*}
	A_\varepsilon(t) e^{-\rho_\varepsilon t} \to W_\varepsilon, \quad \text{ almost surely as } t \to \infty,
	\end{align*}
	where $ \rho_\varepsilon \in \R $ is the (unique) solution to
	\begin{align} \label{def_rho_eps}
	(1-\varepsilon) \int_{0}^{\infty} \overline{\lambda}(t) e^{-\rho_\varepsilon t} dt = 1.
	\end{align}
\end{lemma}

\begin{proof}
	We need to check the conditions of Theorem~3.2 in \cite{crump_general_1969}.
	First, since $ \lambda(t) \leq \lambda^* $, for any $ p > 1 $,
	\begin{align*}
		\int_{0}^{\infty} (\overline{\lambda}(t))^p dt \leq (\lambda^*)^{p-1} \int_{0}^{\infty} \overline{\lambda}(t) dt = (\lambda^*)^{p-1} R_0,
	\end{align*}
	which we have assumed to be finite.
	On the other hand, if $ N $ is the number of offsprings of a given individual, then, using the properties of the Poisson distribution,
	\begin{align*}
		\E \left[ N^2 \right] = \E \left[ \int_{0}^{\infty} \lambda(t) dt \right] + \E \left[ \left( \int_{0}^{\infty} \lambda(t) dt\right)^2 \right]  
		<\infty,
	\end{align*}
	by assumption (this is also true if the individual was initially infected, replacing $ \lambda $ by $ \lambda^0 $ above).
	This concludes the proof.
\end{proof}

\begin{remark}
	The condition $\varepsilon \le 1-\frac{1}{R_0} $ in Theorem \ref{thm:early_phase} ensures that there exists a positive solution $\rho_\varepsilon>0$ to the equation \eqref{def_rho_eps}, i.e., that the branching process 
	$ A_\varepsilon(\cdot) $ is supercritical. This will be used in the proof of Theorem \ref{thm:early_phase}. See also Remark~\ref{rk:eps_R0}.
\end{remark}

\begin{lemma} \label{lemma:rho_eps}
	If $ \rho $ satisfies \eqref{def_rho} and $ \rho_\varepsilon $ is given by \eqref{def_rho_eps}, then, for all $ \varepsilon \in (0,1) $,
	\begin{align*}
	0 \leq \rho - \rho_\varepsilon \leq \frac{\varepsilon}{1-\varepsilon} \left( \int_{0}^{\infty} \overline{\lambda}(t) t e^{-\rho t} dt \right)^{-1}.
	\end{align*}
\end{lemma}

\begin{proof}
	From the definitions of $ \rho $ and $ \rho_\varepsilon $,
	\begin{align*}
	\int_{0}^{\infty} \overline{\lambda}(t) \left( e^{-\rho_\varepsilon t} - e^{-\rho t} \right) dt = \frac{\varepsilon}{1-\varepsilon}.
	\end{align*}
	Hence it is clear that $ \rho \geq \rho_\varepsilon $.
	In addition, $ e^{-\rho_\varepsilon t} - e^{-\rho t} \geq t e^{-\rho t} (\rho-\rho_\varepsilon) $, from which the stated inequality follows.
\end{proof}

\begin{lemma} \label{lemma:continuity_W_eps}
	Let $ (W_\varepsilon, \varepsilon \in [0,1)) $ be the family of random variables defined in Lemma~\ref{lemma:exp_branching}.
	Then
	\begin{align*}
		\lim_{\varepsilon \downarrow 0} \P (W_\varepsilon = 0) = \P (W_0 = 0).
	\end{align*}
\end{lemma}

\begin{proof}
	In \cite{crump_general_1969}, it is shown that $ \P(W_\varepsilon = 0) $ is the probability of extinction of a branching process in which each individual born after time 0 leaves a conditionally Poisson number of offsprings with parameter $ (1-\varepsilon) \int_{0}^{\infty} \lambda(t) dt$.
	Thus if $ X_0 $ denote the random variable corresponding to the number of offsprings of the $ I(0) $ individuals alive at time 0, then
	\begin{align} \label{eq_P_W_eps_0}
		\P(W_\varepsilon = 0) = \E \left[ q_\varepsilon^{X_0} \right],
	\end{align}
	where $ q_\varepsilon $ is the unique fixed point in $ (0,1) $ of the function $ s \mapsto h_\varepsilon(s) $ defined by
	\begin{align*}
		h_\varepsilon(s) := \E \left[ s^{X_\varepsilon} \right],
	\end{align*}
	where $ X_\varepsilon $ is conditionally Poisson with parameter $ (1-\varepsilon) \int_{0}^{\infty} \lambda(t) dt $.
	It is then straightforward to see that $ h_\varepsilon $ converges to $ h_0 $ locally uniformly when $ \varepsilon \downarrow 0 $, and, as a result, $ q_\varepsilon \to q_0 $.
	We then conclude from \eqref{eq_P_W_eps_0} and the dominated convergence theorem.
\end{proof}

We can now prove Theorem~\ref{thm:early_phase}.

\begin{proof}[Proof of Theorem~\ref{thm:early_phase}]
	We begin by a lower bound on $ T^N_\varepsilon $.
	By \eqref{upper_bound_branching}, for any $ \delta \in (0,1) $,
	\begin{align*}
	A^N\left( \frac{1-\delta}{\rho} \log(N) \right) \leq A_0\left( \frac{1-\delta}{\rho} \log(N) \right).
	\end{align*}
	Noting that $ \rho_0 = \rho $, by Lemma~\ref{lemma:exp_branching}, almost surely, for all $ N $ large enough,
	\begin{align*}
	A_0\left( \frac{1-\delta}{\rho} \log(N) \right) \leq N^{1-\delta} (W_0 + \delta).
	\end{align*}
	But $ N^{1-\delta}(W_0 + \delta) < \varepsilon N $ for $ N $ large enough.
	It follows that, for any $ \delta \in (0,1) $,
	\begin{align} \label{lower_bound_T_eps}
	\liminf_{N \to \infty} \frac{T^N_\varepsilon}{\log(N)} \geq \frac{1-\delta}{\rho}, \quad \text{ almost surely.}
	\end{align}
	By the same argument, for any $ \delta \in (0,\alpha) $ and $ \alpha \in (0,1) $,
	\begin{align} \label{bound_inf_T_alpha}
	\liminf_{N \to \infty} \frac{\mathcal{T}^N_\alpha}{\log(N)} \geq \frac{\alpha - \delta}{\rho}.
	\end{align}
	
	On the event $ \lbrace W_0 = 0 \rbrace $, the branching process $ (A_0, \mathfrak{I}_0) $ goes extinct (\textit{i.e.}, $ \mathfrak{I}_0(t) = 0 $ for all $ t $ large enough), and
	\begin{align*}
	\lim_{t \to \infty} A_0(t) < +\infty.
	\end{align*}
	As a result, for any $ t > 0 $,
	\begin{align*}
	A^N\left( t \log(N) \right) &\leq A_0(t\log(N)) \\
	&\leq \lim_{s \to \infty} A_0(s).
	\end{align*}
	Hence $ \mathcal{T}^N_\alpha > t \log(N) $ for all $ t > 0 $ for all $ N $ such that $ N^\alpha > \lim_{t \to \infty} A_0(t) $.
	Hence
	\begin{align}\label{bound_inf_T_alpha_ext}
	\liminf_{N \to \infty} \frac{\mathcal{T}^N_\alpha}{\log(N)} = +\infty,
	\end{align}
	almost surely on the event $ \lbrace W_0 = 0 \rbrace $ for any $ \alpha \in (0,1) $.
	Since $ T^N_\varepsilon \geq \mathcal{T}^N_\alpha $ for $ \alpha \in (0,1) $ and $ N $ large enough, we also obtain
	\begin{align} \label{T_esp_ext}
	\liminf_{N \to \infty} \frac{T^N_\varepsilon}{\log(N)} = + \infty,
	\end{align}
	almost surely on the same event.
	
	We now prove the upper bound on $ \mathcal{T}^N_\alpha $ on the event $ \lbrace W_0 > 0 \rbrace $.
	By Lemma~\ref{lemma:coupling}, for any $ \delta \in (0,1-\alpha) $ and $ \varepsilon \in (0,1/2) $, for $ N $ large enough,
	\begin{align*}
	A^N\left( \frac{\alpha + \delta}{\rho} \log(N) \wedge T^N_\varepsilon \right) \geq A_{2\varepsilon} \left( \frac{\alpha + \delta}{\rho} \log(N) \wedge T^N_\varepsilon \right).
	\end{align*}
	By \eqref{lower_bound_T_eps}, $ T^N_\varepsilon \geq \frac{\alpha + \delta}{\rho} \log(N) $ for all $ N $ large enough (choosing a different $ \delta $ in \eqref{lower_bound_T_eps} if needed) and, by Lemma~\ref{lemma:exp_branching},
	\begin{align*}
	A_{2\varepsilon}\left( \frac{\alpha+\delta}{\rho} \log(N) \right) \geq \frac{W_{2\varepsilon}}{2} N^{\frac{\rho_{2\varepsilon}}{\rho}(\alpha + \delta)},
	\end{align*}
	almost surely for $ N $ large enough.
	By Lemma~\ref{lemma:rho_eps}, we can choose $ \varepsilon $ small enough that
	\begin{align*}
	\frac{\rho_{2\varepsilon}}{\rho} (\alpha + \delta) > \alpha.
	\end{align*}
	As a result,
	\begin{align} \label{T_alpha_W_eps}
	\P \left( \left\lbrace \limsup_{N \to \infty} \frac{\mathcal{T}^N_\alpha}{\log(N)} > \frac{\alpha+\delta}{\rho} \right\rbrace \cap \lbrace W_0 > 0 \rbrace \right) \leq \P \left( \lbrace W_{2\varepsilon} = 0 \rbrace \cap \lbrace W_0 > 0 \rbrace \right).
	\end{align}
	Since, by construction, $ A_{2\varepsilon}(t) \leq A_0(t) $,
	\begin{align*}
	\P \left( \lbrace W_{2\varepsilon} = 0 \rbrace \cap \lbrace W_0 > 0 \rbrace \right) = \P (W_0 > 0) - \P(W_{2\varepsilon} > 0).
	\end{align*}
	The right hand side can then be made arbitrarily small by choosing $ \varepsilon $ small enough by Lemma~\ref{lemma:continuity_W_eps}.
	Since the left hand side in \eqref{T_alpha_W_eps} does not depend on $ \varepsilon $, we conclude that
	\begin{align} \label{bound_sup_T_alpha_surival}
	\limsup_{N \to \infty} \frac{\mathcal{T}^N_\alpha}{\log(N)} \leq \frac{\alpha + \delta}{\rho},
	\end{align}
	almost surely on $ \lbrace W_0 > 0 \rbrace $.
	Combining \eqref{bound_inf_T_alpha}, \eqref{bound_inf_T_alpha_ext} and \eqref{bound_sup_T_alpha_surival}, we obtain that, for any $ \alpha \in (0,1) $, almost surely,
	\begin{align*}
	\frac{\mathcal{T}^N_\alpha}{\log(N)} \to \begin{cases}
	\frac{\alpha}{\rho} & \text{ if } W_0 > 0 \\
	+\infty & \text{ otherwise.}
	\end{cases}
	\end{align*}
	This convergence thus holds in distribution for the original model defined in Subsection~\ref{subsec:model}.
	
	We now prove the upper bound on $ T^N_\varepsilon $ on the event $ \lbrace W_0 > 0 \rbrace $ for $ \varepsilon < 1-\frac{1}{R_0} $.
	To do this, we define, for $ \delta \in (0,1) $, $ \varepsilon' \in (\varepsilon, 1-\frac{1}{R_0}) $ and $ \eta \in (0,1) $,
	\begin{align*}
	&\mathfrak{I}^N_-(t) := \sum_{k=1}^{I(0)} \lambda^0_k(t) + \int_{0}^{t} \int_{0}^{\infty} \int_D \lambda(t-s) {\bf 1}_{u \leq q^N(s) \mathfrak{I}^N_-(s^-)} Q(ds, du, d\lambda),\\
	&A^N_-(t) := \int_{0}^{t} \int_{0}^{\infty} \int_D {\bf 1}_{u \leq q^N(s) \mathfrak{I}^N_-(s^-)} Q(ds, du, d\lambda),
	\end{align*}
	where
	\begin{align*}
	q^N(t) = \begin{cases}
	1-\eta & \text{ if } 0 \leq t \leq \frac{1-\delta}{\rho} \log(N) \\
	1-\varepsilon' & \text{ otherwise.}
	\end{cases}
	\end{align*}
	We note that, for $ t \leq \frac{1-\delta}{\rho} \log(N) $, $ (\mathfrak{I}^N_-(t), A^N_-(t)) = (\mathfrak{I}_\eta(t), A_\eta(t)) $ and, by a similar argument as in Lemma~\ref{lemma:coupling}, for all $ N $ large enough, using \eqref{lower_bound_T_eps},
	\begin{align} \label{coupling_A_minus}
	\forall t \leq T^N_\varepsilon, \quad \mathfrak I^N(t) \geq \mathfrak{I}^N_-(t), \quad A^N(t) \geq A^N_-(t).
	\end{align}
	In addition, for any $ \delta' > 0 $,
	\begin{align*}
	A^N_-\left( \frac{1+\delta'}{\rho} \log(N) \right) = A_\eta\left( \frac{1-\delta}{\rho} \log(N) \right) \frac{A^N_-\left( \frac{1-\delta}{\rho} \log(N) + \frac{\delta+\delta'}{\rho} \log(N) \right)}{A_\eta\left( \frac{1-\delta}{\rho} \log(N) \right)}.
	\end{align*}
	By Lemma~\ref{lemma:exp_branching}, for all $ N $ large enough
	\begin{align} \label{lower_bound_A_eta}
	A_\eta\left( \frac{1-\delta}{\rho} \log(N) \right) \geq \frac{W_\eta}{2} N^{\frac{\rho_\eta}{\rho}(1-\delta)}.
	\end{align}
	Next we note that we can write, for $ t \geq 0 $,
	\begin{align*}
	A^N_-\left( \frac{1-\delta}{\rho}\log(N) + t \right) = \sum_{i=1}^{A_\eta\left( \frac{1-\delta}{\rho} \log(N) \right)} \tilde{A}_i(t),
	\end{align*}
	where $ (\tilde{A}_i(t), t \geq 0)_{i\ge1} $ is a family of i.i.d. branching processes of the form
	\begin{align*}
		&\tilde{A}_i(t) = \int_{0}^{t} \int_{0}^{\infty} \int_D {\bf 1}_{u \leq \mathfrak (1-\varepsilon') \tilde{I}_i(s^-)} \tilde{Q}_i(ds, du, d\lambda), \\
		&\tilde{\mathfrak I}_i(t) = \tilde{\lambda}^0_i(t) + \int_{0}^{t} \int_{0}^{\infty} \int_D \lambda(t-s) {\bf 1}_{u \leq (1-\varepsilon') \tilde{\mathfrak I}(s^-)} \tilde{Q}_i(ds, du, d\lambda),
	\end{align*}
where $\{Q, \tilde{Q}_1, \tilde{Q}_2,\ldots\}$ are i.i.d., and $Q$ is the PRM which was used in the definition of the branching process $A_\eta$ up to time $\frac{1-\delta}{\rho} \log(N)$.
	Since $ \varepsilon' < 1-\frac{1}{R_0} $, $ \tilde{A}_i $ is supercritical and has growth rate $ \rho_{\varepsilon'} > 0 $.
	Moreover, by Lemma~\ref{lemma:exp_branching}, $ e^{-\rho_{\varepsilon'} t} \tilde{A}_i(t) \to \tilde{W}_i $ as $ t \to \infty $, where the $ \tilde{W}_i $ are i.i.d. and such that $ \P(\tilde{W}_i > 0) > 0 $.
	As a result, on $ \lbrace W_{\eta} > 0 \rbrace $, from \eqref{lower_bound_A_eta},
	\begin{align*}
		A_\eta\left( \frac{1-\delta}{\rho} \log(N) \right) \to \infty
	\end{align*}
	and, by the law of large numbers, as $ N \to \infty $,
	\begin{align*}
	\frac{A^N_-\left( \frac{1-\delta}{\rho} \log(N) + \frac{\delta+\delta'}{\rho} \log(N) \right)}{A_\eta\left( \frac{1-\delta}{\rho} \log(N) \right)} N^{-\frac{\rho_{\varepsilon'}}{\rho}(\delta + \delta')} \to \E [ \tilde{W}_1 ] > 0.
	\end{align*}
	Hence on the event $ \lbrace W_\eta > 0 \rbrace $, for some constant $ C > 0 $ and for $ N $ large enough,
	\begin{align*}
	A^N_-\left( \frac{1+\delta'}{\rho} \log(N) \right) \geq \frac{C\, W_\eta}{4} N^{\frac{\rho_\eta}{\rho}(1-\delta) + \frac{\rho_{\varepsilon'}}{\rho}(\delta + \delta')}.
	\end{align*}
	But by Lemma~\ref{lemma:rho_eps}, for any $ \delta'> 0 $ and $ \varepsilon' < 1-\frac{1}{R_0} $ (which ensures that $ \rho_{\varepsilon'} > 0 $), we can choose $ \eta $ and $ \delta $ small enough that
	\begin{align*}
	\frac{\rho_\eta}{\rho}(1-\delta) + \frac{\rho_{\varepsilon'}}{\rho}(\delta + \delta') > 1.
	\end{align*}
	For such a choice of $ \eta $ and $ \delta $,
	\begin{align*}
	A^N_-\left( \frac{1+\delta'}{\rho} \log(N) \right) > N
	\end{align*}
	for all $ N $ large enough, almost surely on the event $ \lbrace W_\eta > 0 \rbrace $.
	By \eqref{coupling_A_minus}, this implies
	\begin{align*}
	\P \left( \left\lbrace \limsup_{N \to \infty} \frac{T^N_\varepsilon}{\log(N)} > \frac{1+\delta'}{\rho} \right\rbrace \cap \lbrace W_0 > 0 \rbrace \right) \leq \P(W_0 > 0) - \P(W_\eta > 0),
	\end{align*}
	for all $ \eta > 0 $ small enough.
	Letting $ \eta \to 0 $ and using Lemma~\ref{lemma:continuity_W_eps}, we thus obtain
	\begin{align*}
	\limsup_{N \to \infty} \frac{T^N_\varepsilon}{\log(N)} \leq \frac{1+\delta'}{\rho},
	\end{align*}
	almost surely on $ \lbrace W_0 > 0 \rbrace $, for any $ \delta' > 0 $.
	Combining this with \eqref{lower_bound_T_eps} and \eqref{T_esp_ext} yields the result.
\end{proof}

Let us now prove Theorem~\ref{thm:early_phase_deterministic}. 

\begin{proof}[Proof of Theorem~\ref{thm:early_phase_deterministic}]
	Plugging \eqref{det_solution_exponential} into \eqref{eqlin0}, and replacing $ \overline{\lambda}^0 $ and $ F_0 $ by $ \overline{\lambda}_\rho $ and $ F_\rho $, we obtain
	\begin{align*}
		I(0) \overline{\lambda}^0(t) + \int_{0}^{t} \overline{\lambda}(t-s) \mathfrak I(s) ds &= \int_{0}^{\infty} \overline{\lambda}(t+s) \rho e^{-\rho s} ds + \int_{0}^{t} \overline{\lambda}(t-s) \rho e^{\rho s} ds \,,\\
		I(0) F^c_0(t) + \int_{0}^{t} F^c(t-s) \mathfrak I(s) ds &= \int_{0}^{\infty} F^c(t+s) \rho e^{-\rho s} ds + \int_{0}^{t} F^c(t-s) \rho e^{\rho s} ds\,.
	\end{align*}
	Changing variables in each integral and then summing them together, we obtain
	\begin{align*}
		\int_{0}^{\infty} \overline{\lambda}(t+s) \rho e^{-\rho s} ds + \int_{0}^{t} \overline{\lambda}(t-s) \rho e^{\rho s} ds &= \int_{t}^{\infty} \overline{\lambda}(s) \rho e^{\rho(t-s)} ds + \int_{0}^{t} \overline{\lambda}(s) \rho e^{\rho (t-s)} ds \\
		&= \rho e^{\rho t},
	\end{align*}
	where we have used \eqref{def_rho} in the last line.
	The same calculation with $ F^c $ instead of $ \overline{\lambda} $ yields
	\begin{align*}
		\int_{0}^{\infty} F^c(t+s) \rho e^{-\rho s} ds + \int_{0}^{t} F^c(t-s) \rho e^{\rho s} ds = \int_{0}^{\infty} F^c(s) \rho e^{\rho (t-s)} ds = \bm{i} e^{\rho t},
	\end{align*}
	using the definition of $ \bm{i} $ in \eqref{def_i_r}.
	In the case $ \rho < 0$, these calculations are unchanged, and we simply multiply each line by $ -1 $.
	Finally, the equation for $ R(t) $ follows from the fact that
	\begin{align*}
		I(t) + R(t) &= I(0) + R(0) + \int_{0}^{t} \mathfrak I(s) ds \\
		&= R(0) + I(0) + \int_{0}^{t} |\rho| e^{\rho s} ds.
	\end{align*}
	Subtracting $ I(t) = |\bm{i}| e^{\rho t} $, we obtain
	\begin{align*}
		R(t) = R(0) + \text{sign}(\rho)(1-\bm{i}) (e^{\rho t}-1).
	\end{align*}
	Since $ \bm{r} = 1-\bm{i} $, this concludes the proof (we choose $ R(0) = \bm{r} $ in the case $ \rho > 0 $).
\end{proof}

\section{Proof of the FLLN}\label{sec:proofFLLN}

In this section, for a sequence $\{X^N,N\ge1\}$ of random elements of $D$, and $X$ a random element of $D$, $X^N\Rightarrow X$ in $D$ means that $X^N$ converges weakly (i.e., in law) towards $X$ in $D$, that is, for 
any $\Phi\in C_b(D;\RR)$, $\E[\Phi(X^N)]\to\E[\Phi(X)]$ as $N\to\infty$.
\subsection{Convergence of $(\bar{S}^N, \bar{\mathfrak{I}}^N)$.}
For the process $A^N(t)$, we have the decomposition
\begin{align} \label{eqn-An-decomp}
A^N(t)= M_A^N(t) + \int_0^t\Upsilon^N(s)ds,
\end{align}
where
$$
M_A^N(t) = \int_0^t\int_0^\infty{\bf1}_{u\le \Upsilon^N(s^-)}\overline{Q}(ds,du), 
$$
with $\overline{Q}(ds,du) = Q(ds,du)- dsdu$ being the compensated PRM.
It is clear that the process $\{M_A^N(t): t\ge 0\}$ is a square-integrable martingale (see, e.g., \cite[Chapter VI]{ccinlar2011probability}) with respect to the filtration 
$\{\sF^N_t: t\ge 0\}$ defined by 
$$
\sF^N_t :=\sigma\bigg\{E^N(0), I^N(0), \{\lambda^0_j(\cdot)\}_{j\ge1},\{\lambda^{0,I}_k(\cdot)\}_{k\ge1}, \{\lambda_i(\cdot)\}_{i\ge1}, \int_0^{t'}\int_0^\infty{\bf1}_{u\le \Upsilon^N(s^-)}Q(ds,du): t' \le t \bigg\}.
$$
It has a finite quadratic variation
$$
\langle M_A^N\rangle(t) =  \int_0^t\Upsilon^N(s)ds, \quad t \ge 0. 
$$
Under Assumption \ref{AS-lambda}, we have
\begin{equation} \label{eqn-int-Phi-bound}
0 \le N^{-1} \int_s^{t} \Upsilon^N(u) du \le \lambda^*(t-s),  \quad \text{w.p.\,1} \qforq 0 \le s \le t.
\end{equation}
Thus, this implies that,  in probability as $N\to\infty$,
$$
\langle \overline{M}_A^N\rangle(t) = N^{-2} \int_0^t\Upsilon^N(s)ds \to 0  \qinq D,
$$
and by Doob's inequality, 
\begin{equation} \label{eqn-barMn-conv}
\overline{M}^N_A(t) \to 0   
\end{equation}
in mean square, locally uniformly in $t$, hence in probability in $D$.
As a consequence, we obtain the following lemma. 
\begin{lemma} \label{lem-barS-tight}
Under Assumptions \ref{AS-lambda}, \ref{AS-FLLN} and \ref{indep}, 
the sequence $\{(\bar{A}^N, \bar{S}^N)\}_{N\ge1}$ is tight in $D^2$.  The limit of any converging subsequence of $\{\bar{A}^N\}$, denoted by $\bar{A}$, satisfies
\begin{equation}
\bar{A}= \lim_{N\to\infty} \bar{A}^N = \lim_{N\to\infty}  \int_0^{\cdot} \bar\Upsilon^N(u) du, 
\end{equation}
and 
\begin{equation}
0 \le \bar{A}(t) - \bar{A}(s) \le \lambda^*(t-s),  \quad \text{w.p.\,1} \qforq 0 \le s \le t.
\end{equation}
Given the limit $\bar{A}$ of a converging subsequence of $\{\bar{A}^N\}$, along the same subsequence, $\bar{S}^N\Rightarrow \bar{S} := \bar{S}(0) - \bar{A} = 1- \bar{I}(0) - \bar{A}$ in $D$ as $N\to \infty$. 
\end{lemma}

Let 
\begin{align*}
 \bar{\mathfrak{I}}^N_{0,1}(t) := N^{-1} \sum_{k=1}^{I^N(0)} \lambda^{0,I}_k(t), \quad \bar{\mathfrak{I}}^N_{0,2}(t) := N^{-1} \sum_{j=1}^{E^N(0)} \lambda^0_j(t),  \quad t \ge 0. 
\end{align*}
\begin{lemma}\label{le:J0}
Under Assumptions  \ref{AS-lambda} and \ref{AS-FLLN}, as $N\to\infty$,
\begin{align} \label{eqn-barfrakI0-conv}
\big(\bar{\mathfrak{I}}^N_{0,1}, \bar{\mathfrak{I}}^N_{0,2} \big)
\to \big( \bar{\mathfrak{I}}_{0,1}, \bar{\mathfrak{I}}_{0,2} \big) \qinq D^2 \ \text{in probability,}
\end{align}
 where 
\begin{align*}
 \bar{\mathfrak{I}}_{0,1}(t):= \bar{I}(0) \bar\lambda^{0,I}(t), \quad \bar{\mathfrak{I}}_{0,2}(t):= \bar{E}(0) \bar\lambda^0(t),  \quad t\ge 0. 
\end{align*}
\end{lemma}
\begin{proof}
Define the processes 
\begin{align}\label{eqn-tilde-frakI-0}
 \widetilde{\mathfrak{I}}^N_{0,1}(t) := N^{-1} \sum_{k=1}^{N \bar{I}(0)} \lambda^{0,I}_k(t), \quad \widetilde{\mathfrak{I}}^N_{0,2}(t) := N^{-1} \sum_{j=1}^{N \bar{E}(0)} \lambda^0_j(t),  \quad t \ge 0. 
\end{align}
By the i.i.d. assumptions for the sequences  $\{\lambda^0_j(t)\}$ and 
 $\{\lambda^{0,I}_k(t)\}$, and their independence,  and by the LLN for random elements in $D$ (see Theorem 1 in \cite{rao1963law} or Corollary 7.10 in \cite{ledoux2013probability}), 
 we directly obtain that, as $N\to\infty$,
 \begin{align*}
\big(\widetilde{\mathfrak{I}}^N_{0,1}, \widetilde{\mathfrak{I}}^N_{0,2} \big)
\to \big( \bar{\mathfrak{I}}_{0,1}, \bar{\mathfrak{I}}_{0,2} \big) \qinq D^2 \ \text{ in probability.} 
\end{align*}

It then suffices to show that, as $N \to \infty$,
 \begin{align} \label{eqn-barfrak-0-diff-conv}
\big(\widetilde{\mathfrak{I}}^N_{0,1}-\bar{\mathfrak{I}}^N_{0,1}, \widetilde{\mathfrak{I}}^N_{0,2} - \bar{\mathfrak{I}}^N_{0,2} \big)
\to 0 \qinq D^2 \ \text{  in probability.} 
\end{align}
We have 
\begin{align} \label{eqn-barfrak-I-0-diff}
\widetilde{\mathfrak{I}}^N_{0,1}(t)-\bar{\mathfrak{I}}^N_{0,1}(t) = \text{sign}(\bar{I}(0) - \bar{I}^N(0)) N^{-1} \sum_{k=N (\bar{I}^N(0) \wedge \bar{I}(0) )}^{N (\bar{I}^N(0) \vee \bar{I}(0) )} \lambda^{0,I}_k(t),
\end{align}
and thus 
\begin{align*}
\sup_{0\le t\le T} \big|\widetilde{\mathfrak{I}}^N_{0,1}(t)-\bar{\mathfrak{I}}^N_{0,1}(t) \big| \le  \lambda^{\ast}  \big|\bar{I}^N(0) -\bar{I}(0)\big|. 
\end{align*}
By the convergence $\bar{I}^N(0) -\bar{I}(0) \to 0$ in probability under Assumption \ref{AS-FLLN}, we obtain that $\widetilde{\mathfrak{I}}^N_{0,1}-\bar{\mathfrak{I}}^N_{0,1}\to 0$ in $D$ in probability. A similar argument yields the convergence $\widetilde{\mathfrak{I}}^N_{0,2}-\bar{\mathfrak{I}}^N_{0,2}\to 0$ in $D$ in  probability. This completes the proof. 
\end{proof}

Let 
\begin{align*}
\bar{\mathfrak{I}}^N_1(t) := N^{-1}\sum_{i=1}^{A^N(t)}\lambda_i (t-\tau^N_i), \quad t\ge 0. 
\end{align*}

Before we prove the convergence of $\bar{\mathfrak{I}}^N_1$ in $D$, let us first establish three technical results which will be useful in the next proof.
 The first of those results was implicitly used in  \cite{PP-2020}.
\begin{lemma}\label{Lem-20}
Let $\{X^N\}_{N\ge1}$ be a sequence of random elements in $D$. If the two conditions
\begin{enumerate}
\item for all $\ep>0$, $0\le t\le T$, $\P\big(|X^N(t)|>\ep\big)\to0$, as $N\to\infty$, and
\item for all $\ep>0$, $\limsup_{N}\sup_{0\le t\le T}\frac{1}{\delta}\P\big(\sup_{0\le u\le \delta}|X^N(t+u)-X^N(t)|>\ep\big)\to0$, as $\delta\to0$
\end{enumerate}
 are satisfied  for all $T>0$, then $X^N(t)\to0$ in probability locally uniformly in $t$.
 \end{lemma}
 \begin{proof}
 We partition the interval $[0,T]$ into subintervals of length $\delta$, that is, we define
 $t_i=i\delta\wedge T$,  $i=0,1,\ldots,\lfloor T/\delta\rfloor$, and obtain
 \begin{align*}
 \sup_{t\in[0,T]} |X^N(t)| \le \sup_{i=0,\ldots,\lfloor T/\delta\rfloor} |X^N(t_{i})| + 
 \sup_{i=0,\ldots,\lfloor T/\delta\rfloor}\sup_{u\in [0,\delta]} |X^N(t_{i}+u) - X^N(t_{i})|\,.
 \end{align*}
 We immediately obtain the following inequality
 \begin{align*}
 \P\left(\sup_{0\le t\le T}|X^N(t)|>\eps\right)
 & \le\sum_{i=0}^{\lfloor T/\Delta\rfloor}\P(|X^N(t_i)|>\eps/2) \\
 & \qquad+\left(\frac{T}{\delta}+1\right)\sup_{0\le t\le T}\P\left(\sup_{0\le u\le \delta}|X^N(t+u)-X^N(t)|>\ep/2\right)\,.
 \end{align*}
  From condition (i), $\limsup_N$ of the first term on the right is zero for any $\delta>0$, while by condition (ii), $\limsup_N$ of the second term tends to zero as $\delta\to0$. The result follows.
 \end{proof}
 
 In the next statement, $D_\uparrow(\R_+)$ (resp. $C_\uparrow(\R_+)$) 
 denotes the set of real-valued nondecreasing
 function on $\R_+$, which belong to $D(\R_+)$ (resp. $C(\R_+)$).
\begin{lemma}\label{le:Portmanteau}
Let $f\in D(\R_+)$ and $\{g_N\}_{N\ge1}$ be a sequence of elements of $D_\uparrow(\R_+)$ which is such that
$g_N\to g$ locally uniformly as $N\to\infty$, where $g\in C_\uparrow(\R_+)$. Then, for any $t>0$, as $N\to\infty$,
\[ \int_{[0,t]} f(s) g_N(ds)\to \int_{[0,t]} f(s)g(ds)\,. \]
\end{lemma}
\begin{proof}
The assumption implies that the sequence of measures $g_N(ds)$ converges weakly, as $N\to\infty$,
towards the measure $g(ds)$. Since, moreover, $f$ is bounded and the set of discontinuities of $f$ is
of $g(ds)$ measure $0$,  the convergence is essentially a minor improvement of the Portmanteau theorem, see Theorem 2.1 in \cite{billingsley1999convergence}. 
\end{proof}

\begin{lemma}\label{le:convD-C}
Let $\{X^N,\ N\ge1\}$ be a sequence of random elements in $D$, which is such that for all $k\ge1$, $0\le t_1<t_2<\cdots<t_k$, as $N\to\infty$, $(X^N(t_1),\ldots,X^N(t_k))\Rightarrow (X(t_1),\ldots,X(t_k))$, and moreover the sequence $\{X^N\}$ satisfies condition (ii) of Lemma  \ref{Lem-20}. Then $X^N\Rightarrow X$ in $D$, and moreover $X\in C$ a.s. If, in addition, for all $t\ge0$, $X^N(t)\to X(t)$ in probability, then $X^N(t)\to X(t)$ in probability locally uniformly in $t$.
\end{lemma}
\begin{proof}
Define the modulus of continuity on $[0,T]$ of a function $x$ as 
\[\omega_x(T,\delta)=\sup_{0\le s<t\le T,\ t-s\le\delta}|x(t)-x(s)|\,.\]
It is clear (see the proof of Theorem 7.4 in \cite{billingsley1999convergence}) that 
\[\P(\omega_{X^N}(T,\delta)>3\ep)\le\sup_{0\le t\le T}\left(\frac{T}{\delta}+1\right)\P\bigg(\sup_{0\le u\le \delta}|X^N(t+u)-X^N(t)|>\ep\bigg)\]
Since the ``$D$--modulus of continuity'' $\omega'_x(T,\delta)$ satisfies  $\omega'_x(T,\delta)\le\omega_x(T,2\delta)$ (see (12.7) in \cite{billingsley1999convergence}), we conclude from Theorem 13.2 and its Corollary in \cite{billingsley1999convergence} that $\{X^N\}$ is tight in $D$. Since all finite dimensional distributions of $X^N$ converge to those of $X$, all converging subsequences of the sequence $\{X^N\}$ converge to $X$, and the whole sequence converges to $X$. Moreover, it follows from our assumptions that for any $T>0$, $\omega_X(T,\delta)\to0$, as $\delta\to0$, hence $X\in C$ a.s. Concerning the convergence in probability, we note that under the additional assumption, $Y^N(t):=X^N(t)-X(t)$ satisfies the conditions of Lemma \ref{Lem-20}, hence the result.
\end{proof}

\begin{lemma} \label{lem-barfrakI}
Under Assumptions \ref{AS-lambda} and \ref{indep}, if $\bar{A}$ is the limit of a converging subsequence of $\{\bar{A}^N\}$, then along the same subsequence, 
\begin{align} \label{eqn-barfrakI-conv}
\bar{\mathfrak{I}}^N_1 
\Rightarrow \bar{\mathfrak{I}}_1 \qinq D \qasq N \to \infty,
\end{align} 
where 
\begin{align*}
\bar{\mathfrak{I}}_1(t):= \int_0^t \bar\lambda (t-s) d \bar{A}(s), \quad t\ge 0. 
\end{align*}

\end{lemma}

\begin{proof}
Let 
\begin{align} \label{eqn-brevefrakI}
\breve{\mathfrak{I}}^N_1(t) := 
N^{-1}\sum_{i=1}^{A^N(t)}\bar\lambda (t-\tau^N_i) = \int_0^t \bar\lambda (t-s) d \bar{A}^N(s), \quad t \ge 0. 
\end{align}
The proof will be split in two steps.
\bigskip

{\sc Step 1. Convergence of $\breve{\mathfrak{I}}^N_1$}

Under Assumption \ref{AS-lambda},  applying  Lemmas \ref{lem-barS-tight} and \ref{le:Portmanteau} and the continuous mapping theorem, we obtain that, as $N\to\infty$, all finite dimensional distributions of 
$\breve{\mathfrak{I}}^N_1$ converge to those of $\bar{\mathfrak{I}}_1$.
It remains to establish condition (ii) from Lemma \ref{Lem-20} in order to deduce from Lemma \ref{le:convD-C} that
\begin{align} \label{eqn-brevefrakI-conv}
\breve{\mathfrak{I}}^N_1 
\Rightarrow \bar{\mathfrak{I}}_1 \qinq D \qasq N \to \infty.
\end{align} 
That is, we need to show that
\begin{equation}\label{eq:cvchech}
\lim_{\delta \to 0} \limsup_{N\to \infty} \frac{1}{\delta} \sup_{t \in [0,T]} \P \left(\sup_{u \in [0,\delta]}\big| \breve{\mathfrak{I}}^N_1(t+u) - \breve{\mathfrak{I}}^N_1(t) \big|> \ep \right) =0.
\end{equation}
We have for $t, u \ge 0$,
\begin{align*}
\big| \breve{\mathfrak{I}}^N_1(t+u) - \breve{\mathfrak{I}}^N_1(t) \big| 
&\le  \left| N^{-1}\sum_{i=1}^{A^N(t)}\big( \bar\lambda(t+u-\tau^N_i) - \bar\lambda(t-\tau^N_i) \big)    \right| \non\\
& \quad +  N^{-1}  \sum_{i=A^N(t)+1}^{A^N(t+u)} \bar\lambda(t+u-\tau^N_i)   \\
& =:\Delta^{N,1}_{t,u}+\Delta^{N,2}_{t,u}. 
\end{align*}
 We first note that by \eqref{eqn-int-Phi-bound},
 \begin{align*}
\sup_{0 \le u \le \delta} \Delta^{N,2}_{t,u}
&  \le \lambda^* \big( \bar{A}^N(t+\delta) - \bar{A}^N(t)\big) \\
& \le   (\lambda^*)^2 \delta + \lambda^* \big(\bar{M}^N_A(t+\delta)-\bar{M}^N_A(t)\big), 
 \end{align*}
so that by \eqref{eqn-barMn-conv},   for any $T>0$, $\ep>0$, provided $\delta<\eps/(4(\lambda^\ast)^2)$,
\begin{align*} 
\P \bigg(\sup_{0 \le u \le \delta}  \Delta^{N,2}_{t,u} > \ep/2 \bigg)& \le  \P\left(\left|\bar{M}^N_A(t+\delta)-\bar{M}^N_A(t)\right|>\eps/4\lambda^\ast\right)\\
&\to0,\ \text{ as }N\to\infty,
\end{align*}
and consequently,
\begin{equation}\label{eqn-Vn-diff-3}
\limsup_{N\to \infty} \frac{1}{\delta} \sup_{t \in [0,T]} \P \bigg(\sup_{u \in [0,\delta]}\big| \Delta^{N,2}_{t,u} \big|> \ep/2 \bigg) =0.
\end{equation} 
 We now consider the first term $\Delta^{N,1}_{t,u}$.  
 Let
 \[\Lambda_\delta(t):=\sup_{u\le\delta}|\bar\lambda(t+u)-\bar\lambda(t)|\,.\]
We have
\[ \sup_{u\le\delta}\Delta^{N,1}_{t,u}\le\int_0^t\Lambda_\delta(t-s)d\bar A^N(s)\,,\]
and
\begin{align*}
\P\bigg(\sup_{u\le\delta}|\Delta^{N,1}_{t,u}|>\frac{\ep}{2}\bigg)&\le
\P\left(\int_0^t\Lambda_\delta(t-s)d\bar A^N(s)>\frac{\ep}{2}\right)\\
&\le
\P\left(\left|\int_0^t\Lambda_\delta(t-s)d\bar M_A^N(s)\right|>\frac{\ep}{4}\right)+\P\left(\int_0^t\Lambda_\delta(t-s)\bar\Upsilon^N(s)ds>\frac{\ep}{4}\right)\,.
\end{align*}
It is not hard to show that for any $\delta>0$,
\[ \limsup_{N\to +\infty}\frac{1}{\delta}\sup_{t\in[0, T]}\P\left(\left|\int_0^t\Lambda_\delta(t-s)d\bar M_A^N(s)\right|>\frac{\ep}{4}\right)=0\,.\]
Next we note that for any $t\in[0,T]$,
\begin{align*}
\int_0^t\Lambda_\delta(t-s)\bar\Upsilon^N(s)ds&\le\lambda^\ast\int_0^t\Lambda_\delta(t-s)ds\\
&\le\lambda^\ast\int_0^T\Lambda_\delta(s)ds\,.
\end{align*}
Since $\bar\lambda$ is right continuous and bounded by $\lambda^*$, this last expression tends to $0$ as $\delta\to0$. Consequently, for $\delta>0$ small enough, 
\[ \sup_N\sup_{t\in[0,T]}\P\left(\int_0^t \Lambda_\delta(t-s)\bar\Upsilon^N(s)ds>\frac{\ep}{4}\right)=0\,.\]
 It follows that \eqref{eqn-Vn-diff-3} holds true with $\Delta^{N,2}_{t,u}$ replaced by $\Delta^{N,1}_{t,u}$.
We have completed the proof of \eqref{eq:cvchech}, hence of  \eqref{eqn-brevefrakI-conv}. 
\bigskip

{\sc Step 2. $\mathfrak{I}^N_1-\breve{\mathfrak{I}}^N_1\to0$}

Now it remains to show that, as $N\to\infty$, 
\begin{equation} \label{eqn-Vn-conv}
V^N:= \bar{\mathfrak{I}}^N_1 - \breve{\mathfrak{I}}^N_1 \to 0 \qinq D \ \text{ in probability}. 
\end{equation}
We have 
\begin{align*}
V^N(t) = N^{-1} \sum_{i=1}^{A^N(t)} \chi^N_i(t), \quad \chi^N_i(t):= \lambda_i (t-\tau^N_i) - \bar\lambda (t-\tau^N_i).
\end{align*}
 $\chi_i^N(t)$ clearly satisfies $\E\big[\chi^N_i(t)\big]=0$ and $\E\big[\chi^N_i(t)\chi^N_j(t)|\tau^N_i,\tau^N_j]=0$. Thus, 
\begin{align*}
\E\big[V^N(t)^2\big] = N^{-2} \E\Bigg[ \sum_{i=1}^{A^N(t)} \nu(t-\tau^N_i)\Bigg]   = N^{-1}  \E \bigg[ \int_0^t \nu(t-s) d \bar{A}^N(s) \bigg],
\end{align*}
where $\nu(t) := E[ (\lambda_i (t) - \bar\lambda (t))^2]$ and $\nu(t)< (\lambda^*)^2$ under Assumption \ref{AS-lambda}. 
We easily obtain that for each $t\ge0$,
$$
V^N(t)\to 0\  \text{ in probability, } \qasq N \to \infty\, .
$$

It remains to establish condition (ii) of Lemma \ref{Lem-20}, i.e., that for any $T>0$, $\ep>0$, 
\begin{align} \label{eqn-Vn-diff-tight}
\lim_{\delta \to 0} \limsup_{N\to \infty} \frac{1}{\delta} \sup_{t \in [0,T]} \P \left(\sup_{u \in [0,\delta]}\big| V^N(t+u) - V^N(t) \big|> \ep \right) =0.
\end{align}
We have for $t, u \ge 0$,
\begin{align*}
\big| V^N(t+u) - V^N(t) \big| 
&\le  \left| N^{-1}\sum_{i=1}^{A^N(t)}\big( \lambda_i(t+u-\tau^N_i) - \lambda_i(t-\tau^N_i) \big)    \right| \non\\
&\quad +   \left| N^{-1}\sum_{i=1}^{A^N(t)}\big(  \bar\lambda(t+u-\tau^N_i)  - \bar\lambda(t-\tau^N_i) \big)    \right| \non \\
& \quad + \left|  N^{-1}  \sum_{i=A^N(t)+1}^{A^N(t+u)}\big( \lambda_i(t+u-\tau^N_i) - \bar\lambda(t+u-\tau^N_i) \big)   \right|. 
\end{align*}
The second term has already been treated in {\sc Step 1}, and the treatment of the third term is the same as that of the second term in the analogous inequality in {\sc Step 1} in \eqref{eqn-Vn-diff-3}. It remains to treat the first term, which we denote by $\Phi^{N,1}_{t,u}$.
By Assumption~\ref{AS-lambda},
\begin{align*}
\Phi^{N,1}_{t,u}&\le N^{-1}\sum_{i=1}^{A^N(t)}\sum_{j=1}^k|\lambda^j_i(t+u-\tau^N_i)-\lambda^j_i(t-\tau^N_i)|{\bf1}_{\xi_i^{j-1}\le t-\tau^N_i<t+u-\tau^N_i<\xi_i^j}\\&\qquad
+\lambda^\ast N^{-1}\sum_{i=1}^{A^N(t)}\sum_{j=1}^{k}{\bf1}_{ t-\tau^N_i\le\xi_i^j<t+u-\tau^N_i}\\
&\le \varphi_{T+\delta}(u)\bar{A}^N(t)+ \lambda^\ast\sum_{j=1}^{k}N^{-1}\sum_{i=1}^{A^N(t)}{\bf1}_{ t-\tau^N_i\le\xi_i^j<t+u-\tau^N_i}\,. 
\end{align*}
The right hand side being nondecreasing in $u$,  we deduce that
\begin{align*}
\sup_{0\le u\le\delta}\Phi^{N,1}_{t,u}\le \varphi_{T+\delta}(\delta)\bar A^N(t)+\lambda^\ast\sum_{j=1}^{k}N^{-1}\sum_{i=1}^{A^N(t)}{\bf1}_{ t-\tau^N_i\le\xi_i^j<t+\delta-\tau^N_i}\,.
\end{align*}
The first term on the right is the same as the one which appeared in the upper bound of $\Delta^{N,1}_{t,u}$ in 
{\sc Step 1}. We need only consider the second term.
We have
\begin{align} \label{eqn-Vn-diff-p2-1}
& \P \Bigg(\lambda^\ast\sum_{j=1}^{k}N^{-1}\sum_{i=1}^{A^N(t)}{\bf1}_{ t-\tau^N_i\le\xi_i^j<t+\delta-\tau^N_i}> \ep \Bigg)\non\\
&  \le  \frac{1}{\ep^2} \E \Bigg[ \Bigg( \lambda^\ast\sum_{j=1}^{k}N^{-1}\sum_{i=1}^{A^N(t)}{\bf1}_{ t-\tau^N_i\le\xi^j_i<t+\delta-\tau^N_i}\Bigg)^2 \Bigg] \non \\
& \le \frac{2}{\ep^2} \E \Bigg[ \Bigg( \lambda^\ast\sum_{j=1}^{k}N^{-1} \int_0^t  \int_0^\infty \int_{t-s}^{t+\delta-s} \bone_{u \le \Upsilon^N(s^-)} \overline{Q}_j(ds, du, d \xi )
\Bigg)^2 \Bigg] \non  \\
& \quad + \frac{2}{\ep^2} \E \Bigg[ \Bigg( \lambda^\ast\sum_{j=1}^{k} N^{-1}\int_0^t  \big(F_j(t+\delta-s) - F_j(t-s)\big) \Upsilon^N(s)ds \Bigg)^2 \Bigg],
\end{align}
where  $Q_j(ds,du,d\xi)$ is a PRM on $\RR_+\times \RR_+\times \RR_+$ with mean measure $ds du F_j(d\xi)$, and $\overline{Q}_j(ds,du,d\xi)$ is the corresponding compensated PRM. 
Observe that 
\begin{align*} 
\E \left[ \left( N^{-1}\!\! \int_0^t \!\!  \int_0^\infty\!\! \int_{t-s}^{t+\delta-s}\!\!\! \bone_{u \le \Upsilon^N(s^-)} \overline{Q}_j(ds, du, d \xi ) \right)^2 \right]  
& = N^{-2} \E \left[\int_0^t  \big(F_j(t+\delta-s) - F_j(t-s)\big)  \Upsilon^N(s) ds \right] \non\\
& \le N^{-1}\lambda^* \int_0^t  \big(F_j(t+\delta-s) - F_j(t-s)\big) ds, 
\end{align*}
which tends to $0$ as $N\to\infty$, for any $\delta>0$. Moreover, 
\begin{align*} 
  \E \left[ \left(N^{-1}\!\!\int_0^t \!\! \big(F_j(t+\delta-s) - F_j(t-s)\big) \Upsilon^N(s)ds \right)^2 \right] 
  &\le \left(\lambda^*\!\! \int_0^t \!\! \big(F_j(t+\delta-s) - F_j(t-s)\big) ds \right)^2\\
  &\le\left(\lambda^\ast\left(\int_t^{t+\delta}F_j(u)du-\int_0^\delta F_j(u)du\right)\right)^2\\
  &\le(\lambda^\ast\delta)^2\,.
\end{align*}
   We deduce that for any $\ep>0$,
\begin{equation}\label{eqn-Vn-diff-1}
\limsup_{N\to \infty} \frac{1}{\delta} \sup_{t \in [0,T]} \P \bigg(\sup_{u \in [0,\delta]}\big| \Phi^{N,1}_{t,u} \big|> \ep \bigg) \to0, \qasq \delta\to0.
\end{equation}
We have proved \eqref{eqn-Vn-diff-tight}. 
This completes the proof of the lemma. 
\end{proof}

From the proof of Lemma \ref{lem-barfrakI}, clearly $(\bar{A}^N,\breve{\mathfrak{I}}^N_1)\Rightarrow
(\bar{A},\bar{\mathfrak{I}}_1)$ along a subsequence. It also follows from Lemma \ref{le:J0} and the proof of Lemma \ref{lem-barfrakI} that $\bar{\mathfrak{I}}^N-\breve{\mathfrak{I}}^N_1\to\bar{\mathfrak{I}}_{0,1}+\bar{\mathfrak{I}}_{0,2}$ in probability in $D$, as $N\to\infty$.
Hence $(\bar{A}^N,{\mathfrak{I}}^N)\Rightarrow(\bar{A},\bar{\mathfrak{I}})$ along the same subsequence as above, where $\bar{\mathfrak{I}}=\bar{\mathfrak{I}}_{0,1}+\bar{\mathfrak{I}}_{0,2}+\bar{\mathfrak{I}}_1$. It follows that, along that subsequence,
\begin{align} \label{eqn-barPhi-conv}
\int_0^{\cdot} \bar{\Upsilon}^N(s) ds = \int_0^{\cdot} \bar{S}^N(s)\bar{\mathfrak{I}}^N(s)  ds  \Rightarrow  \int_0^{\cdot} \bar{S}(s)\bar{\mathfrak{I}}(s)  ds \qinq D, 
\end{align}
and also
\begin{align} \label{eqn-barAn-conv-A}
\bar{A}^N \Rightarrow \bar{A} =  \int_0^{\cdot} \bar{S}(s)\bar{\mathfrak{I}}(s)  ds \qinq D.
\end{align}
Therefore, the limits $\big(\bar{S},\bar{\mathfrak{I}}\big)$ satisfy the integral equations \eqref{eqn-barS} and \eqref{eqn-barfrakI} in Theorem \ref{thm-FLLN}.  Finally, the existence and uniqueness of a deterministic solution to the integral equations follows from applying Gronwall's inequality in a straightforward way, and the whole sequence converges in probability. This completes the proof of the convergence of $\big(\bar{S}^N,\bar{\mathfrak{I}}^N\big) \to\big(\bar{S},\bar{\mathfrak{I}}\big)$ in $D^2$ in probability.  

\subsection{Convergence of $(\bar{E}^N,\bar{I}^N,  \bar{R}^N)$ }
The proof for the convergence of $(\bar{E}^N, \bar{I}^N, \bar{R}^N)$  will be similar to the previous step.

For the initially exposed and infectious individuals, let
\begin{align*}
&\bar{E}^N_0(t) := N^{-1}\sum_{j=1}^{E^N(0)}{\bf1}_{\zeta^0_j>t}\,, \quad  \bar{I}^N_{0,1}(t) := N^{-1} \sum_{k=1}^{I^N(0)}{\bf1}_{\eta^{0,I}_k>t}\,,  \quad  \bar{I}^N_{0,2}(t) := N^{-1}\sum_{j=1}^{E^N(0)}{\bf1}_{\zeta^0_j + \eta^0_j>t}\,,
\\
&\bar{R}^N_{0,1}(t) := N^{-1} \sum_{k=1}^{I^N(0)}{\bf1}_{\eta^{0,I}_k\le t}\,, \quad \bar{R}^N_{0,2}(t) :=N^{-1}\sum_{j=1}^{E^N(0)}{\bf1}_{\zeta^0_j + \eta^0_j\le t} \,. 
\end{align*}
By the FLLN for empirical processes, we obtain the following lemma. 
\begin{lemma} 
Under Assumption  \ref{AS-FLLN}, as $ N \to \infty$,
\begin{align} \label{eqn-barEIR0-conv}
\big(\bar{E}^N_0, \bar{I}^N_{0,1}, \bar{I}^N_{0,2}, \bar{R}^N_{0,1}, \bar{R}^N_{0,2}\big) \to 
\big(\bar{E}_0, \bar{I}_{0,1}, \bar{I}_{0,2}, \bar{R}_{0,1}, \bar{R}_{0,2}\big)
\qinq D^5 \ \text{in probability,}
\end{align}
where 
\begin{align*}
& \bar{E}_0(t) = \bar{E}(0) G_0^c(t), \quad \bar{I}_{0,1}(t) = \bar{I}(0) F_{0,I}^c(t), \quad \bar{I}_{0,2}(t) = \bar{E}(0)  \Psi_0(t) ,   \\
& \bar{R}_{0,1}(t) = I(0) F_{0,I}(t), \quad \bar{R}_{0,2}(t) = \bar{E}(0) \Phi_0(t).  
\end{align*}
\end{lemma}
\begin{proof}
Recall the definition of $\big(\widetilde{\mathfrak{I}}^N_{0,1}, \widetilde{\mathfrak{I}}^N_{0,2})$ in \eqref{eqn-tilde-frakI-0}. Similarly,  define 
$\big( \widetilde{E}^N_0, \widetilde{I}^N_{0,1}, \widetilde{I}^N_{0,2}, \widetilde{R}^N_{0,1}, \widetilde{R}^N_{0,2}\big)$ by replacing $E^N(0)$ and $I^N(0)$ with $ N \bar{E}(0)$ and $N \bar{I}(0)$, respectively, in the definitions of
$\big(\bar{E}^N_0, \bar{I}^N_{0,1}, \bar{I}^N_{0,2}, \bar{R}^N_{0,1},\\ \bar{R}^N_{0,2}\big)$. 
By the i.i.d. assumption of $\{\lambda^{0,I}_k\}_{k\ge 1}$ and the definition of $\eta^{0,I}_k$ from $\lambda^{0,I}_k$ in \eqref{eqn-zeta-lambda-0I},   we obtain that, as $N\to\infty$, 
$$\big(\widetilde{\mathfrak{I}}^N_{0,1}, \widetilde{I}^N_{0,1}, \widetilde{R}^N_{0,1}\big) \to \big(\bar{\mathfrak{I}}_{0,1}, \bar{I}_{0,1}, \bar{R}_{0,1}\big) \qinq D^3 \ \text{ in probability.} 
$$

  Similarly, by the i.i.d. assumption of $\{\lambda^{0}_j\}_{j\ge 1}$ and the definition of $(\zeta^{0}_j, \eta^0_j)$ from $\lambda^{0}_j$ in \eqref{eqn-zeta-lambda-0},   we obtain that, as $N\to\infty$,
\[\big(\widetilde{E}^N_{0}, \widetilde{I}^N_{0,2}, \widetilde{R}^N_{0,2}\big) \to \big(\bar{E}_{0}, \bar{I}_{0,2}, \bar{R}_{0,2}\big) \qinq D^3 \ \text{ in probability.} 
\]
Then it remains to show that, as $N\to \infty$, 
\[\big(\widetilde{E}^N_0 -\bar{E}^N_0, \widetilde{I}^N_{0,1}- \bar{I}^N_{0,1}, \widetilde{I}^N_{0,2}- \bar{I}^N_{0,2}, \widetilde{R}^N_{0,1} - \bar{R}^N_{0,1}, \widetilde{R}^N_{0,2}-\bar{R}^N_{0,2}\big) \to 
0 
\qinq D^5\ \text{ in probability.}\]
Similarly as in the proof of Lemma \ref{le:J0}, we have
\begin{align*}
\widetilde{I}^N_{0,2}(t) - \bar{I}^N_{0,2}(t) = \text{sign}(\bar{E}(0) - \bar{E}^N(0)) N^{-1} \sum_{j=N (\bar{E}^N(0) \wedge \bar{E}(0)) }^{N (\bar{E}^N(0) \vee \bar{E}(0)) }\bone_{\zeta_j^0+\eta_j^0>t},
\end{align*}
and 
\begin{align*}
\E\Bigg[ N^{-1} \sum_{j=N (\bar{E}^N(0) \wedge \bar{E}(0)) }^{N (\bar{E}^N(0) \vee \bar{E}(0)) }\bone_{\zeta_j^0+\eta_j^0>t} \, \Bigg| \, \sF^N_0 \Bigg]\le  \Psi_0(t) |\bar{E}(0) - \bar{E}^N(0)|\to 0\qasq N \to \infty.
\end{align*}
The other convergences follow by a similar argument.  
This completes the proof. \end{proof}

For the newly infected individuals, let 
\begin{align*}
& \bar{E}^N_1(t) := N^{-1} \sum_{i=1}^{A^N(t)}{\bf1}_{\tau^N_i+\zeta_i>t}\,, \quad \bar{I}^N_1(t) := N^{-1} \sum_{i=1}^{A^N(t)}{\bf1}_{\tau^N_i+\zeta_i\le t<\tau^N_i+\zeta_i+\eta_i}\,,  \\
& \bar{R}^N_1(t) := N^{-1} \sum_{i=1}^{A^N(t)}{\bf1}_{\tau^N_i+\zeta_i+\eta_i\le t}\,.  
\end{align*}

\begin{lemma}\label{le:final}
Under Assumptions \ref{AS-lambda} , \ref{AS-FLLN} and \ref{indep}, as $N \to \infty$,
\begin{align} \label{eqn-barEIR-conv}
\big(\bar{E}^N_1, \bar{I}^N_1, \bar{R}^N_1 \big)
\to \big(\bar{E}_1,\bar{I}_1, \bar{R}_1 \big) \qinq D^3 \ \text{ in probability,}
\end{align}
where 
\begin{align*}
 \bar{E}_1(t) &:=\int_0^tG^c(t-s)\bar{S}(s)\bar{\mathfrak{I}}(s)ds\,,\quad 
 \bar{I}_1(t) :=  \int_0^t \Psi(t-s)\bar{S}(s)\bar{\mathfrak{I}}(s)ds\,, \\
 \bar{R}_1(t) &:=  \int_0^t \Phi(t-s)\bar{S}(s)\bar{\mathfrak{I}}(s)ds\,. 
\end{align*}
\end{lemma}

\begin{proof}
We first note that we have the two identities $\bar A^N(t)=\bar{E}^N_1(t) +\bar{I}^N_1(t) +\bar{R}^N_1(t) $ and $\bar A(t)=\bar{E}_1(t) +\bar{I}_1(t) +\bar{R}_1(t) $, which  reflects the two facts:
\begin{align*}
1&=\bone_{\zeta_i\le t-\tau^N_i<\zeta_i+\eta_i}+\bone_{\zeta_i>t-\tau^N_i}+\bone_{\zeta_i+\eta_i\le t-\tau^N_i},\\
1&=\Psi(t-s)+G^c(t-s)+\Phi(t-s)\,.
\end{align*}
Consequently, since we already know that $\bar A^N(t)\to\bar A(t)$ in probability locally uniformly in $t$, we only need to establish the two convergences $\bar{E}^N_1\to\bar{E}_1$ and $\bar{R}^N_1\to\bar{R}_1$, from which the convergence $\bar{I}^N_1\to\bar{I}_1$ will follow as a corollary.

We shall apply the same argument as in Lemma \ref{lem-barfrakI}, but now we know that $\bar{A}^N\to \bar{A}$ in probability. Define
\begin{align*}
\breve{E}^N_1(t) &:= N^{-1} \sum_{i=1}^{A^N(t)} G^c(t- \tau^N_i) = \int_0^t G^c(t-s) d \bar{A}^N(s)\,, \\
 \breve{R}^N_1(t) & := N^{-1} \sum_{i=1}^{A^N(t)}\Phi(t- \tau^N_i) = \int_0^t \Phi(t-s) d \bar{A}^N(s)\,. 
\end{align*}
 Let us establish that $\bar{E}^N_1\to\bar{E}_1$. We shall then discuss why the same arguments work in the case of $\bar{R}^N_1$.
 
 {\sc Step 1}
It follows from Lemma \ref{le:Portmanteau} that for all $t>0$, $\breve{E}^N_1(t)\to\bar{E}_1(t)$ in probability. In order to 
establish that the convergence is in fact locally uniform in $t$, according to Lemma \ref{le:convD-C}, it remains to prove that condition (ii) in Lemma \ref{Lem-20} is satisfied, namely that
\begin{equation}\label{eq:cvchechE}
\lim_{\delta \to 0} \limsup_{N\to \infty} \frac{1}{\delta} \sup_{t \in [0,T]} \P \left(\sup_{u \in [0,\delta]}\big| \breve{E}^N_1(t+u) - \breve{E}^N_1(t) \big|> \ep \right) =0.
\end{equation}
We have
\begin{align*}
\breve{E}^N_1(t+u) - \breve{E}^N_1(t)&=\int_0^t[G^c(t+u-s)-G^c(t-s)]d\bar A^N(s)+\int_t^{t+u}G^c(t+u-s)d\bar A^N(s)\,, \\
\sup_{0<u\le\delta}|\breve{E}^N_1(t+u) - \breve{E}^N_1(t)|&\le
\int_0^t[G^c(t-s)-G^c(t+\delta-s)]d\bar A^N(s)+ \bar A^N(t+\delta)-\bar A^N(t)\,.
\end{align*}
The second term in the right hand side satisfies
\[   \bar A^N(t+\delta)-\bar A^N(t)\le\lambda^\ast\delta+ \bar M_A^N(t+\delta)-\bar M_A^N(t),\]
and since $\bar M_A^N$ tends to $0$ locally uniformly in $t$, 
\[ \limsup_N\sup_{t\in[0,T]}\frac{1}{\delta}\P\big(\bar A^N(t+\delta)-\bar A^N(t)>\ep\big)=0,\]
as soon as $\delta<\ep/\lambda^\ast$. Moreover
\begin{align*}
\P\left(\int_0^t[G^c(t-s)-G^c(t+\delta-s)]d\bar A^N(s)>\ep\right)&\le
\P\left(\left|\int_0^t[G^c(t-s)-G^c(t+\delta-s)]d\bar M_A^N(s)\right|>\ep/2\right)\\
&\quad+\P\left(\int_0^t[G^c(t-s)-G^c(t+\delta-s)]\bar\Upsilon^N(s)ds>\ep/2\right)\,.
\end{align*}
It is not hard to show that for any $\delta>0$,
\[ \limsup_N\frac{1}{\delta}\sup_{t\in[0, T]}\P\left(\left|\int_0^t[G^c(t-s)-G^c(t+\delta-s)]d\bar M_A^N(s)\right|>\ep/2\right)=0\,.\]
Next we note that for any $t\in[0,T]$,
\begin{align*}
\int_0^t[G^c(t-s)-G^c(t+\delta-s)]\bar\Upsilon^N(s)ds&\le\lambda^\ast\int_0^t[G^c(s)-G^c(s+\delta)]ds\\
&\le\lambda^\ast\int_0^T[G^c(s)-G^c(s+\delta)]ds\,.
\end{align*}
Since $G^c$ is right continuous and bounded by $1$, this last expression tends to $0$ as $\delta\to0$. Consequently, for $\delta>0$ small enough, 
\[ \sup_N\sup_{t\in[0,T]}\P\left(\int_0^t[G^c(t-s)-G^c(t+\delta-s)]\bar\Upsilon^N(s)ds>\ep/2\right)=0\,.\]
Thus, \eqref{eq:cvchechE} has been established, hence $\breve{E}^N_1(t)\to\bar{E}_1(t)$ in probability locally uniformly in $t$.
It remains to consider $\bar E^N_1-\breve E^N_1$, which we do in the next step.

{\sc Step 2}
Consider
\[ W^N(t):=\bar E^N_1(t)-\breve E^N_1(t)=\frac{1}{N}\sum_{i=1}^{A^N(t)}
\big(\bone_{\zeta_i>t-\tau^N_i}-G^c(t-\tau^N_i)\big)\,.\]
It is not hard to see that if $i\not=j$,
\[\E\Big[\big(\bone_{\zeta_i>t-\tau^N_i}-G^c(t-\tau^N_i)\big)\big(\bone_{\zeta_j>t-\tau^N_j}-G^c(t-\tau^N_j)\big)
\Big|\tau^N_i,\tau^N_j\Big]=0\,.\] 
Consequently, 
\begin{align*}
\E\big[\left(W^N(t)\right)^2\big]&=\frac{1}{N^2}\E \Bigg[\sum_{i=1}^{A^N(t)}G^c(t-\tau^N_i)(1-G^c(t-\tau^N_i)) \Bigg] \\
&=\frac{1}{N}\E\bigg[\int_0^tG^c(t-s)(1-G^c(t-s))d\bar A^N(s) \bigg] \\
&\to 0,\quad\text{as }N\to\infty\,.
\end{align*}
It remains to show that condition (ii) of Lemma \ref{Lem-20} holds, namely that
\begin{align} \label{eqn-Wn-diff-tight}
\lim_{\delta \to 0} \limsup_{N\to \infty} \frac{1}{\delta} \sup_{t \in [0,T]} \P \left(\sup_{u \in [0,\delta]}\big| W^N(t+u) - W^N(t) \big|> \ep \right) =0.
\end{align}
We have
\begin{align*}
|W^N(t+u) - W^N(t)|&\le\frac{1}{N}\sum_{i=1}^{A^N(t)}\big(\bone_{\zeta_i>t-\tau^N_i}-\bone_{\zeta_i>t+u-\tau^N_i}\big)+\frac{1}{N}\sum_{i=1}^{A^N(t)}\left(G^c(t-\tau^N_i)-G^c(t+u-\tau^N_i)\right)\\
&\quad+\left|\frac{1}{N}\sum_{i=A^N(t)+1}^{A^N(t+u)}\big(\bone_{\zeta_i>t+u-\tau^N_i}-G^c(t+u-\tau^N_i)\big)\right|\,.
\end{align*}
The second term has already been treated in {\sc Step 1}, as well as $\bar A^N(t+\delta)-\bar A^N(t)$, which bounds the third term. It remains to treat the first term. Let 
\begin{align*}
\Delta^N_1(t,u):&=\frac{1}{N}\sum_{i=1}^{A^N(t)}\bone_{t-\tau^N_i<\zeta_i\le t+u-\tau^N_i},\\
\sup_{u\le \delta}\Delta^N_1(t,u)&=\frac{1}{N}\sum_{i=1}^{A^N(t)}\bone_{t-\tau^N_i<\zeta_i\le t+\delta-\tau^N_i}\,,\\
\P\bigg(\sup_{u\le \delta}\Delta^N_1(t,u)>\ep\bigg)&
\le\frac{1}{\ep^2}\E\Bigg[\Bigg(\frac{1}{N}\sum_{i=1}^{A^N(t)}\bone_{t-\tau^N_i<\zeta_i\le t+\delta-\tau^N_i}\Bigg)^2\Bigg]\,.
\end{align*}
Let $P(ds,du,d\zeta)$ be a PRM on $\RR_+\times\RR_+\times\RR_+$ with mean measure $dsduG(d\zeta)$, and $\bar{P}$ the associated compensated measure. We have
\begin{align*}
\E\Bigg[\Bigg(\frac{1}{N}\sum_{i=1}^{A^N(t)}\bone_{t-\tau^N_i<\zeta_i\le t+\delta-\tau^N_i}\Bigg)^2\Bigg]
&=
\E\left[\left(\frac{1}{N}\int_0^t\int_0^\infty\int_{t-s}^{t+\delta-s}\bone_{u\le\Upsilon^N(s^-)}P(ds,du,d\zeta)\right)^2\right]
\\
&\le
2\E\left[\left(\frac{1}{N}\int_0^t\int_0^\infty\int_{t-s}^{t+\delta-s}\bone_{u\le\Upsilon^N(s^-)}\bar{P}(ds,du,d\zeta)\right)^2\right]\\
&\quad+2\E\left[\left(\frac{1}{N}\int_0^t(G^c(t-s)-G^c(t+\delta-s))\Upsilon^N(s)ds\right)^2\right]\,.
\end{align*}
The first term is of ordre $N^{-1}$, and tends to $0$ as $N\to\infty$. The second term is bounded by 
$2(\lambda^\ast)^2$ times
\begin{align*}
 \left(\int_0^t(G(t+\delta-s)-G(t-s))ds\right)^2&\le\left(\int_t^{t+\delta}G(u)du-\int_0^\delta G(u)du\right)^2\\
 &\le\delta^2\,.
 \end{align*}
Consequently
\[\limsup_N\frac{1}{\delta}\sup_{t\le T}\P\left(\sup_{u\le \delta}\Delta^N_1(t,u)>\ep\right)\to0,\ \text{ as }\delta\to0\,.\]

{\sc Step 3. The case of $\bar R^N_1$.}
Essentially the same argument will work in the case of $\bar R^N_1$ ($G^c$ was decreasing, $\Phi$ is increasing).
The details are left to the reader.
\end{proof}
\begin{remark}
A proof of Lemma \ref{le:final} can be found in \cite{PP-2020}. There the authors use the fact that the integral of $G^c(t-s)$ (resp. $\Phi(t-s)$) can be integrated by parts, since $G^c$ (resp. $\Phi$) is decreasing (resp. increasing), thus simplifying step 1 of the proof. However, the present version of step 1, which follows the same argument as Lemma \ref{lem-barfrakI}, allows to shorten step 2. 
\end{remark}

\paragraph{\bf Acknowledgement}
The authors want to thank two anonymous Referees, whose criticisms and suggestions on a first version of this work have led to significant improvements, in particular to the addition of the analysis of the stochastic model during the early phase, namely Theorem \ref{thm:early_phase}.
G. Pang was supported in part by  the US National Science Foundation grant DMS-1715875 and Army Research Office grant W911NF-17-1-0019. 

\bibliographystyle{plain}

\bibliography{Epidemic-Age,bib_raphael}

\end{document}